\documentclass{article}
\usepackage[utf8]{inputenc}

\usepackage{amssymb, amsmath, amsthm}
\usepackage{color}
\usepackage{float}

\usepackage{epsfig}
\usepackage{subfigure}
\usepackage[margin=2.5cm]{geometry}

\newtheorem{theorem}{Theorem}[section]
\newtheorem*{theorem*}{Theorem}
\newtheorem{corollary}[theorem]{Corollary}
\newtheorem*{corollary*}{Corollary}
\newtheorem{proposition}[theorem]{Proposition} 
\newtheorem*{proposition*}{Proposition} 
\newtheorem{remark}[theorem]{Remark} 
\newtheorem{definition}[theorem]{Definition} 
\newtheorem{assumption}[theorem]{Assumption} 

\title{Summation-by-parts approximations of the second derivative: Pseudoinverses of singular operators and revisiting the sixth order accurate narrow-stencil operator}

\author{Sofia Eriksson\thanks{Department of Mathematics, Linnaeus University, Växjö, Sweden. Email: sofia.eriksson@lnu.se} \and Siyang Wang\thanks{Division of Applied Mathematics, UKK, Mälardalen University, Västerås, Sweden. Email: siyang.wang@mdh.se}}
\date{}

\begin{document}

\maketitle

\begin{abstract}
We consider finite difference approximations of the second derivative, exemplified in Poisson’s equation, the heat equation and the wave equation. The  finite difference operators satisfy a summation-by-parts property, which mimics the integration-by-parts. Since the operators approximate the second derivative, they are singular by construction. To impose boundary conditions, these operators are modified using Simultaneous Approximation Terms. This makes the modified matrices non-singular, for most choices of boundary conditions. Recently, inverses of such matrices were derived. However, when considering Neumann boundary conditions on both boundaries, the modified matrix is still singular. For such matrices, we have derived an explicit expression for the Moore--Penrose pseudoinverse, which can be used for solving elliptic problems and some time-dependent problems. The condition for this new pseudoinverse to be valid, is that the modified matrix does not have more than one zero eigenvalue. 
We have reconstructed the sixth order accurate narrow-stencil operator with a free parameter and show that more than one zero eigenvalue can occur. We have performed a detailed analysis on the free parameter to improve the properties of the second derivative operator. We complement the derivations by numerical experiments to demonstrate the improvements of the new second derivative operator. 

\end{abstract}



\newcommand{\trans}{\mathsf{T}}

\section{Introduction}

Partial differential equations (PDEs) that involve the second derivative include Poisson's equation, the heat equation and the wave equation. 
Here, we approximate the second derivative using high-order accurate finite differences. Especially for hyperbolic problems with sufficiently smooth solutions, it is well-known that high-order methods are computationally more efficient than low-order methods \cite{Kreiss1972}. 
In particular, high-order finite difference operators satisfying a summation-by-parts (SBP) property \cite{ref:KREI74} lead to an energy-stable discretization given a suitable numerical boundary treatment, for example the simultaneous-approximation-term (SAT) method \cite{Carpenter1994}. 

The SBP-SAT finite difference methods have been widely used to solve hyperbolic and parabolic PDEs \cite{Del2014Review,Svard2014}. The spatial approximation can also be used to solve time-independent problems, for example Poisson's equation of elliptic type, where the discretization leads to a linear system that needs to be solved. However, for problems with Neumann boundary conditions only, the discretization matrix is singular. In this paper, we derive an analytical formula for the Moore--Penrose inverse of the singular discretization matrix, which can be used to solve Poisson's equation with Neumann boundary conditions. 

It is often assumed 
that the singular discretization matrix is inherently rank-deficient by one. By using recent results from \cite{InversesEriksson2020preprint}, we 
prove that this is indeed the case for the second and fourth order accurate dicretization. In the sixth order case, we find no such proof. In contrast, we find a counter example.
By revisiting the conditions for an SBP-operator, we construct a  one-parameter family of sixth order accurate, narrow-stencil SBP operators for the second derivative. The free parameter was first mentioned in \cite{Sjogreen2012} but no such operator was presented. A particular choice of the free parameter reduces the operator to the traditional sixth order SBP operator derived in \cite{Mattsson2004503}. 
It turns out that it is possible to choose this free parameter such that the discretization matrix is rank-deficient by two or three, even though all the accuracy and SBP stability properties are fulfilled.

Avoiding that unfortunate choice, we instead have the possibility to improve the second derivative operator.
We perform a careful analysis of our operator,
first making sure that it leads to stability, correct nullspace and in addition compatibility with the first derivative operator. We then consider
hyperbolic, parabolic and elliptic problems with Dirichlet or Neumann boundary conditions,
focusing on accuracy and spectral radius. 
With respect to those 
properties, recommendations of the free parameter are given, which depend on the underlying problem. 

The paper is organized as follows. In Section 2, we formulate the problem and introduce the SBP concepts. The Moore--Penrose inverse of the singular discretization matrix (for the Neumann boundary conditions) is derived in Section 3.
In Section 4, we relate the existence of the Moore--Penrose inverse to the rank of the discretization operator.
The new sixth order accurate SBP operators with a free parameter is constructed and analyzed in Section 5.
In Section 6, we perform numerical experiments to investigate how the free parameter affects the properties of the SBP operator.
The paper is summarized in Section 7.


\section{Problem formulation}

\newcommand{\trunc}{\mathbf{r}}
\newcommand{\smooth}{w}
\newcommand{\restrict}{\mathbf{w}}

\newcommand{\ones}{\mathbf{1}}
\newcommand{\xfat}{\mathbf{x}}
\newcommand{\indexl}{\text{\tiny L}}
\newcommand{\indexr}{\text{\tiny R}}
\newcommand{\indexlr}{\text{\tiny L,R}}
\newcommand{\el}{\mathbf{e}_\indexl}
\newcommand{\er}{\mathbf{e}_\indexr }
\newcommand{\elr}{\mathbf{e}_\indexlr}
\newcommand{\dsel}{\mathbf{d}_\indexl}
\newcommand{\dser}{\mathbf{d}_\indexr }
\newcommand{\dselr}{\mathbf{d}_\indexlr}
\newcommand{\app}{\gamma}
\newcommand{\smalll}{\text{l}}
\newcommand{\smallr}{\text{r}}
\newcommand{\mul}{\mu_\indexl}
\newcommand{\mur}{\mu_\indexr}
\newcommand{\mulr}{\mu_\indexlr}

\newcommand{\DD}{D_{\text{\tiny D}}}
\newcommand{\DN}{D_{\text{\tiny N}}}
\newcommand{\DDN}{D_{\text{\tiny D,N}}}
\newcommand{\wDDN}{\widetilde{D}_{\text{\tiny D,N}}}
\newcommand{\wDD}{\widetilde{D}_{\text{\tiny D}}}
\newcommand{\wDN}{\widetilde{D}_{\text{\tiny N}}}

\newcommand{\fD}{\mathbf{f}_{\text{D}}}
\newcommand{\fN}{\mathbf{f}_{\text{N}}}
\newcommand{\fDN}{\mathbf{f}_{\text{D,N}}}

In this section, we present the preliminaries of  notations and basic properties of the SBP-SAT operators. Thereafter we present the continuous problems and their discrete counterparts utilizing these SBP-SAT operators.

\subsection{Preliminaries}
Let $\Omega=[0,1]$ denote a bounded domain in $\mathbb{R}$. We discretize $x\in\Omega$ by using $n+1$ equidistant grid points $x_i=ih$, where $i=0,1,\cdots,n$ and $h=1/n$. We will frequently use the following grid functions
\begin{align}
 \ones=[1, 1, \hdots\ 1]^\trans,\quad &\xfat=[x_0, x_1, \hdots, x_n]^\trans,\label{onesandxfat}\\
 \el=[1, 0, \hdots, 0]^\trans,\quad &\er=[ 0, \hdots, 0, 1]^\trans.\notag
\end{align}
Throughout the paper, bold symbols are used for denoting column vectors. Moreover, 
uppercase letters are reserved for matrices and lowercase letters for scalars.

Let $\smooth(x)\in C^{\infty}(\Omega)$ be a smooth function in $\Omega$ and let $\restrict=\smooth(\xfat)$ be the restriction of $\smooth(x)$ on the grid $\xfat$ in \eqref{onesandxfat}.
The SBP operator $D_1$, that approximates the first derivative such that
$D_1 \restrict \approx \smooth_x(\xfat)$, was constructed in \cite{ref:KREI74,ref:STRA94} and fulfills the SBP property in Definition~\ref{defD1} below.
\begin{definition}\label{defD1}
The first derivative operator $D_1$ is an SBP operator if it satisfies $D_1 = H^{-1} Q$, where $H$ is symmetric positive definite and where {\normalfont $Q+Q^\trans=-\el\el^\trans+\er\er^\trans$}. 
\end{definition}

The symmetric positive definite operator $H$ defines a discrete norm, and it is also a quadrature \cite{Hicken2013}. When $H$ is diagonal, the corresponding operators are called diagonal-norm SBP operators. In this case, $D_1$ has order of accuracy $2p$ in the interior rows and at most order $p$ in a few rows near the boundaries. The accuracy near the boundaries can be improved to order $2p-1$ by using a nondiagonal SBP norm, but the resulting operators are not widely used because energy stability cannot be proved for general problems with variable coefficients.

To approximate the second derivative, we need an SBP operator $D_2$ such that 
$D_2 \restrict \approx \smooth_{xx}(\xfat)$, satisfying the following definition: 
\begin{definition}\label{defD2}
The second derivative operator $D_2$ is an SBP operator if it satisfies {\normalfont 
\begin{equation}\label{SBP_D2}
 D_2=H^{-1}(-A-\el\dsel^\trans+\er\dser^\trans),
\end{equation}}%
where $H$ is symmetric positive definite and $A$ is symmetric positive semi-definite. The operators {\normalfont $\dselr$} are consistent approximations of the first derivative at the left and right boundaries. 
\end{definition}
The most straightforward way to construct a second derivative SBP operator is to apply $D_1$ twice and obtain the {\it wide-stencil} operator $D_1^2$ \cite{Carpenter1999341}, which fulfills the above SBP property. However, the wide-stencil operators have at most order $p-1$ in a few rows near the boundaries. In addition, they produce spurious oscillations for nonsmooth problems. These two shortcomings can be overcome by using {\it narrow-stencil} operators, which are operators with a minimal stencil width in the interior that were constructed in \cite{Mattsson2004503}. 
The narrow-stencil operators have order of accuracy $2p$ in the interior rows and at most order $p$ in a few rows near the boundaries.
%
%
That is, $D_2\restrict=\restrict ^{(2)}+\mathcal{O}(h^q)\restrict^{(q+2)}$, where $q=2p$ for the interior rows and $q=p$ for a few rows near boundaries.  The exact expressions of $\mathcal{O}(h^q)$ are obtained using Taylor expansions, and for consistency $q>0$ is necessary. Replacing the general $\smooth(x)$ by monomials $\smooth=x^k$, the consistency demand can be expressed as
\begin{align}\label{consistency}
 D_2\ones=0,&&D_2\xfat=0,&&\dselr^\trans\ones=0,&&\dselr^\trans\xfat=1.
\end{align}
For $q\geq 1$, the operator has to fulfill $D_2\xfat^k=k(k-1)\xfat^{k-2}$ for $k=2,3,\hdots,q+1$, where  $\xfat^k=[x_0^k, x_1^k, \hdots, x_n^k]^\trans$.
For later reference, \eqref{consistency} includes demands on $\dselr$ as well.
Moreover, when referring to an operator of for example "sixth order", we mean the interior order $2p=6$, unless stated otherwise.

The properties of the matrix $A$ associated with $D_2$ are important. A so-called "borrowing technique" is needed for an energy stable discretization of the wave equation with Dirichlet boundary conditions \cite{APPELO2007531,Mattsson2009} or material interface conditions \cite{MATTSSON20088753}, and for a dual-consistent discretization of the heat equation with Dirichlet boundary conditions \cite{Eriksson2018}. 
Below, we give a definition of the 
{\it borrowing capacity} related to $D_2$.
\begin{definition}\label{bp}
The borrowing capacity $\app$ is the maximum value such that
\[
\widetilde{A} = A-h\app (\dsel\dsel^\trans+\dser\dser^\trans)
\]
is symmetric positive semi-definite.
\end{definition}
The precise value of $\app$ was computed in \cite{InversesEriksson2020preprint,MATTSSON20088753,Virta2014}.
In this paper, our main focus is the second derivative operator, but when solving PDEs with both first and second derivatives, it is important that the SBP norm $H$ in $D_1$ and $D_2$ from Definition~\ref{defD1} and \ref{defD2} are the same. In addition, $D_1$ and $D_2$ must be {\it compatible} \cite{Mattsson2008b}. 
\begin{definition}\label{DefComp}
The SBP operators $D_1$ and $D_2$ are compatible if 
\[
R = A - D_1^\trans H D_1
\]
is symmetric positive semi-definite.
\end{definition}
The SBP operators $D_2$ constructed in \cite{Mattsson2004503} are compatible with $D_1$ constructed in \cite{ref:KREI74,ref:STRA94}.

\subsection{The continuous problem}

\newcommand{\contsol}{u}
\newcommand{\force}{f}
\newcommand{\gNl}{g_\smalll}
\newcommand{\gNr}{g_\smallr }
\newcommand{\gDl}{g_\indexl}
\newcommand{\gDr}{g_\indexr }

Consider the one-dimensional wave equation 
\begin{align}
\label{wave}
\begin{aligned}
\contsol_{tt}&=\contsol_{xx}+\force ,&&x\in[0, 1],\\
\end{aligned}
\end{align}
with suitable initial conditions, 
 where $\force(x,t)$ is
the forcing function. The boundary conditions are either of
Dirichlet type
\begin{align}
\label{waveD}
\contsol(0,t)=\gDl,\quad 
\contsol(1,t)=\gDr,
\end{align}
or of Neumann type
\begin{align}
\label{waveN}
\contsol_{x}(0,t)=\gNl,\quad 
\contsol_x(1,t)=\gNr.
\end{align}
We assume that the initial and boundary data are compatible, sufficiently smooth functions.

The SBP-SAT discretization of the Neumann problem was derived in \cite{Mattsson2004503}, and then
later for the Dirichlet problem  \cite{APPELO2007531,Mattsson2009}. We state them below and also discuss how they can be used for solving stationary problems.

\subsection{The discrete problem}\label{sec_discrete}

\newcommand{\discsol}{\mathbf{v}}
\newcommand{\fh}{\mathbf{f}}
\newcommand{\Fsnok}{\widehat{\mathbf{f}}}
\newcommand{\rhs}{\mathbf{b}}

\newcommand{\sigmal}{\sigma_\indexl}
\newcommand{\sigmar}{\sigma_\indexr}
\newcommand{\sigmalr}{\sigma_\indexlr}
\newcommand{\taul}{\tau_\indexl}
\newcommand{\taur}{\tau_\indexr}
\newcommand{\taulr}{\tau_\indexlr}

\newcommand{\AtildeD}{\widehat{A}_D}
\newcommand{\AtildeN}{\widehat{A}_N}

\newcommand{\Abeg}{a_\indexl}
\newcommand{\Aend}{a_\indexr }
\newcommand{\Aoff}{a_{\text{\tiny C}}}
\newcommand{\Abar}{\bar{A}}
\newcommand{\Awest}{\vec{a}_\indexl}
\newcommand{\Aeast}{\vec{a}_\indexr }
\newcommand{\Avecs}{\vec{a}_\indexlr}

\newcommand{\factor}{\varphi}

First, we consider the Dirichlet problem \eqref{wave} with \eqref{waveD}. The semi-discrete approximation can be written as 
\begin{align}\label{waveD_semi}
\discsol_{tt}&=D_2\discsol+\fh +H^{-1}( \mul \el -\dsel) \left( \el^\trans\discsol-\gDl\right)+H^{-1}(\mur \er + \dser) \left( \er^\trans\discsol-\gDr\right),
\end{align}
where $\discsol$ is the approximation of the continuous solution $\contsol(x,t)$ on the grid $\xfat$, and $\fh$ is the forcing function $\force(x,t)$ evaluated on the grid. By using the SBP identity \eqref{SBP_D2}, we rewrite \eqref{waveD_semi} as 
\begin{align}\label{SchemeWaveDirichlet}
\discsol_{tt}&=\DD\discsol+\fD,
\end{align}
where $\fD=\fh -H^{-1}( \mul \el -\dsel) \gDl-H^{-1}(\mur \er + \dser) \gDr$ and where
\begin{align}\label{D2tildeDirichlet}
 \DD=D_2+H ^{-1}( \mul \el -\dsel) \el^\trans+H ^{-1}(\mur \er + \dser) \er^\trans.
\end{align}
The scheme \eqref{waveD_semi} is energy stable \cite{Mattsson2009} if
\begin{align*}\begin{split}
\mulr&\leq-\frac{1}{h\app},\end{split}
\end{align*}
where $\app$ is the borrowing capacity in Definition \ref{bp}. Even though the choice $\mulr=-1/(h\app)$ yields energy stability, it is not an advisable choice, since it makes the discretization matrix $\DD$ singular, see \cite{InversesEriksson2020preprint} or \cite{WangKreiss2017}. Instead, we are going to use
\begin{align}\label{factor}
\mulr&=-\frac{\factor}{h\app},&&\factor>1,
\end{align}
where the factor $\factor$ needs to be sufficiently large to avoid the risk of sub-optimal convergence \cite{WangKreiss2017}, but not too large either -- since that might lead to stiffness (this aspect will be discussed later in Section~\ref{InfluenceOfFree}).

Next, we consider the Neumann problem \eqref{wave} with \eqref{waveN}. By  using the SBP identity \eqref{SBP_D2}, the semi-discretization 
\begin{align}
\label{waveN_semi}
\begin{split}
\discsol_{tt}=D_2\discsol+\fh &+H^{-1} \sigmal \el \left(\dsel^\trans \discsol-\gNl\right)
+H^{-1}\sigmar \er \left( \dser^\trans \discsol-\gNr\right),
\end{split}
\end{align}
 can be rewritten as
\begin{align}\label{SchemeTime}
\discsol_{tt}&=\DN\discsol+\fN,
\end{align}
where $\fN=\fh -H ^{-1} \sigmal \el\gNl-H ^{-1}\sigmar \er \gNr$ and
 \begin{align}\label{D2tildeNeumann}
\DN&=D_2+H ^{-1} \sigmal \el\dsel^\trans +H ^{-1}\sigmar \er \dser^\trans. 
\end{align}
Given the choice $\sigmal=1$, $\sigmar=-1$, the discretization \eqref{waveN_semi} is energy stable with $\DN=-H^{-1}A$ \cite{Mattsson2004503}.

 We note that the same discretization matrices $\DD$ and $\DN$ can be used to solve the heat equation with the Dirichlet and Neumann boundary conditions, respectively, and the resulting schemes are dual-consistent \cite{Eriksson2018}. In addition, $\DD$ and $\DN$ can also be used to discretize Poisson's equation. In this case, a system of linear equations 
 \begin{align}\label{PoissonBoth}
-\wDDN\discsol = H\fDN
 \end{align}
 with the symmetric matrix $\wDDN=H\DDN$ must be solved. For the Dirichlet problem, the parameter $\mulr$ can be chosen such that $\wDD$ is invertible and the analytical expression of its inverse is derived in \cite{InversesEriksson2020preprint}. However, for the Neumann problem, the corresponding matrix $\wDN=-A$ is always singular. In this case, some pseudoinverse of $A$ must be computed. In the next section, we derive an analytical expression for the Moore-Penrose pseudoinverse of $A$.
\section{The pseudoinverse of $A$}

Consider the one-dimensional Poisson's equation $-\contsol_{xx}=\force$, with Neumann boundary conditions \eqref{waveN}. The solution is only determined up to a constant. To obtain a unique solution,  an additional constraint needs to be added, for example that $\int_\Omega\contsol \text{d}x=0$.
The SBP-SAT discretization is given in \eqref{PoissonBoth}, with $\wDN=-A$ and $\fN$,
that is
\begin{equation}\label{Avb}
A\discsol=\rhs,
\end{equation}
where $\rhs = H\fN$.
Recall that $A$ is singular. When $\rhs$ is in the column space of $A$, the linear system is under-determined and has infinitely many solutions. Corresponding to the continuous case, the  constraint to obtain a unique solution is for example the mean value of  $\discsol$ is zero. If $\rhs$ is not in the column space of $A$, then no solution satisfies  the linear system exactly and a least-squares solution can be computed. In either case, we would like to find a pseudoinverse of $A$. 



\newcommand{\Ibar}{\bar{I}}
\newcommand{\ett}{\vec{1}}
\newcommand{\xvec}{\vec{x}}

Let the parts of $A$ be denoted as shown below 
\begin{align}\label{Aparts}
A=\left[\begin{array}{ccc}\Abeg&\Awest^\trans&\Aoff\\\Awest&\Abar&\Aeast\\\Aoff&\Aeast^\trans&\Aend\end{array}\right],
\end{align} 
where $\Abeg$, $\Aoff$, $\Aend$ are scalars and $\Awest$, $\Aeast$ are $(n-1)\times 1$ vectors. Assuming that $\Abar$ is non-singular, we define  
\begin{align}\label{G2}
G_2=\left[\begin{array}{ccc}0&0&0\\0&\Abar^{-1}&0\\0&0&0\end{array}\right].
\end{align} 
Using the consistency
restrictions  on $D_2$ and $\dselr$ in \eqref{consistency}, we see 
that $A$  in \eqref{SBP_D2} fulfills
\begin{align}\label{A1Ax}
A\ones=0,&&A\xfat=\er-\el. 
\end{align}
The two relations in \eqref{A1Ax} can also be expressed 
componentwise as
\begin{align}\label{A1Axcomponents}
\left[\begin{array}{ccc}\Abeg+\Awest^\trans\ett+\Aoff\\\Awest+\Abar\ett+\Aeast\\\Aoff+\Aeast^\trans\ett+\Aend\end{array}\right]
=\left[\begin{array}{c}0\\0\\0\end{array}\right],&&
\left[\begin{array}{c}\Awest^\trans\xvec+ \Aoff\\\Abar\xvec+ \Aeast\\\Aeast^\trans\xvec+ \Aend\end{array}\right]=\left[\begin{array}{c}-1\\0\\1\end{array}\right].
\end{align}
where $\ett=[1\ 1\ 1\ \hdots\ 1]^\trans$ and $\xvec=[x_1\ x_2\ \hdots\ x_{n-1}]^\trans$ are shorter versions of $\ones$ and $\xfat$ in \eqref{onesandxfat}.
In addition, following the derivations in \cite{InversesEriksson2020preprint} we also obtain the identity 
\begin{align}\label{AG}\begin{split}
AG_2
&=\left[\begin{array}{ccc}0&\Awest^\trans\Abar^{-1}&0\\0&\Ibar&0\\0&\Aeast^\trans\Abar^{-1}&0\end{array}\right]
=\left[\begin{array}{ccc}1&0&0\\0&\Ibar&0\\0&0&1\end{array}\right]+\left[\begin{array}{ccc}-1&(\xvec-\ett)^\trans&0\\0&0&0\\0&-\xvec^\trans&-1\end{array}\right]
=I-\el(\ones-\xfat )^\trans-\er\xfat^\trans, 
\end{split}
\end{align}
where $I$ is the $(n+1)\times(n+1)$ identity matrix and $\Ibar$ is the $(n-1)\times(n-1)$ identity matrix.
In the second step we have used the relations $\Abar^{-1}\Awest+\ett+\Abar^{-1}\Aeast=0$ and $\xvec+ \Abar^{-1}\Aeast=0$ from the mid rows in \eqref{A1Axcomponents}.

We are now ready to present an explicit formula of the Moore–Penrose pseudoinverse of $A$.

\subsection{The Moore–Penrose inverse}
For the matrix $A$ with real entries, the Moore–Penrose inverse $A^+$ fulfills the following four properties \cite{Wangguorong2018}
\begin{subequations}\label{MPeqs}
\begin{align}
\label{MP1}
AA^+A&=A\\
\label{MP2}
A^+AA^+&=A^+\\
\label{MP3}
(AA^+)^\trans&=AA^+\\
\label{MP4}
(A^+A)^\trans&=A^+A.
\end{align}
\end{subequations}
For the SBP matrix $A$ in particular, we have derived the explicit form of $A^+$ shown in Theorem~\ref{ThmMP} below:
\begin{theorem}\label{ThmMP}
Consider $A$ from \eqref{SBP_D2}. If $G_2$  in \eqref{G2} exists, then the Moore–Penrose inverse of $A$ is 
\begin{align}\label{MPinvofA}
A^+=\left(I-\frac{\ones\ones^\trans}{n+1}\right)G_2\left(I-\frac{\ones\ones^\trans}{n+1}\right)+\left(\xfat- \frac{\ones}{2}\right)\left(\xfat- \frac{\ones}{2}\right)^\trans,
\end{align}
where 
$\ones$ and $\xfat$ are given in \eqref{onesandxfat}, and $I$ is the $(n+1)\times (n+1)$ identity matrix.
\end{theorem}

\begin{proof}[Proof of Theorem~\ref{ThmMP}]
It is straightforward to prove the theorem by inserting \eqref{MPinvofA} directly into \eqref{MPeqs} and verify that all four properties are satisfied. In the following, however, we use a more pedagogical approach and present how we have derived the analytical  expression of $A^+$.

Based on \eqref{AG}, we first make the ansatz $A^+=G_2+X$, where $X$ is to be determined. After inserting the ansatz into \eqref{MP1}, we use \eqref{AG} and \eqref{A1Ax} to obtain
\begin{align*}
AA^+A
&=A -(\er-\el)(\er-\el)^\trans +AXA.
\end{align*}
 We see that $X$ must be chosen such that the term $(\er-\el)(\er-\el)^\trans$ is cancelled. 
In \eqref{A1Ax}, we note that $A\xfat=\er-\el$, and can thus specify $X$ further as $X=\xfat\xfat^\trans+\ones \mathbf{y}_1^\trans+\mathbf{y}_2\ones^\trans$. This leads to $AA^+A=A$ for any $(n+1)\times1$-vectors $\mathbf{y}_{1,2}$, since $A\ones=0$. That is, \eqref{MP1} is satisfied if $A^+=G_2+\xfat\xfat^\trans+\ones \mathbf{y}_1^\trans+\mathbf{y}_2\ones^\trans$, and $\mathbf{y}_{1,2}$ will be determined by the other three properties.

We now consider the third requirement \eqref{MP3}, which requires that $AA^+$ is symmetric. We have
\begin{align*}
 AA^+&=A\left(G_2+\xfat\xfat^\trans+\ones \mathbf{y}_1^\trans+\mathbf{y}_2\ones^\trans\right)\\
 &=I-\el\ones ^\trans+A\mathbf{y}_2\ones^\trans,
\end{align*}
where we have used \eqref{AG} and \eqref{A1Ax}. For $AA^+$ to be symmetric, we need to find a $\mathbf{y}_2$ such that $A\mathbf{y}_2-\el =\kappa\ones$ for some constant $\kappa$.
To do that, we first observe that $AG_2$ from \eqref{AG} gives
\begin{align*}
AG_2\ones&=\ones-\el(\ones-\xfat )^\trans\ones-\er\xfat^\trans \ones=\ones-\frac{n+1}{2}\el-\frac{n+1}{2}\er,
\end{align*}
where we have used that $\ones^\trans\ones=n+1$ and that $\xfat^\trans \ones=h(0+1+2+\hdots n)=(n+1)/2$. Combining this with $A\xfat=\er-\el$ from \eqref{A1Ax}, we make the terms containing $\er$ cancel, and obtain
\begin{align*}
\frac{1}{n+1}AG_2\ones+ \frac{1}{2}A\xfat=\frac{1}{n+1}\ones-\el.
\end{align*}
If we choose $\mathbf{y}_2=c_2\ones-\frac{1}{n+1}G_2\ones- \frac{1}{2}\xfat$ with an arbitrary scalar $c_2$, we obtain $A\mathbf{y}_2-\el=-\frac{1}{n+1}\ones$ and consequently $AA^+=I-\frac{1}{n+1}\ones\ones^\trans$, which is indeed symmetric. 

The fourth requirement \eqref{MP4} states that $A^+A$ must be symmetric. With a similar derivation as above, we find that the choice $\mathbf{y}_1=c_1\ones-\frac{1}{n+1}G_2\ones- \frac{1}{2}\xfat$, with any scalar $c_1$, yields the symmetric $A^+A =I -\frac{1}{n+1}\ones\ones^\trans$.
 
So far, after demanding \eqref{MP1}, \eqref{MP3} and \eqref{MP4} to be satisfied, we have
\begin{align*}
A^+&=G_2+\xfat\xfat^\trans+\ones \mathbf{y}_1^\trans+\mathbf{y}_2\ones^\trans\\
&=G_2+\xfat\xfat^\trans-\ones \left(\frac{G_2\ones}{n+1}+ \frac{\xfat}{2}\right)^\trans-\left(\frac{G_2\ones}{n+1}+ \frac{\xfat}{2}\right)\ones^\trans+c\ones\ones^\trans
\end{align*}
where $c=c_1+c_2$.

Last, we consider the condition \eqref{MP2}. Using that $AA^+=I-\frac{1}{n+1}\ones\ones^\trans$ and then inserting the above expression of $A^+$ into $A^+AA^+$  yield
\begin{align*}
A^+AA^+&=A^+\left(I-\frac{1}{n+1}\ones\ones^\trans\right)\\
&=A^+-\frac{G_2+\xfat\xfat^\trans-\ones \left(\frac{G_2\ones}{n+1}+ \frac{\xfat}{2}\right)^\trans-\left(\frac{G_2\ones}{n+1}+ \frac{\xfat}{2}\right)\ones^\trans+c\ones\ones^\trans}{n+1}\ones\ones^\trans\\
&=A^++\left(\frac{\ones^\trans G_2\ones}{(n+1)^2}+ \frac{1}{4}-c\right)\ones\ones^\trans,
\end{align*}
where we have used that $\ones^\trans\ones=n+1$, that
$\xfat^\trans \ones=\frac{n+1}{2}$ and that $\ones^\trans G_2\ones$ is a scalar. 
To satisfy \eqref{MP2}, we need $c=\frac{\ones^\trans G_2\ones}{(n+1)^2}+ \frac{1}{4}$, that is
\begin{align*}
A^+&=G_2+\xfat\xfat^\trans-\ones \left(\frac{G_2\ones}{n+1}+ \frac{\xfat}{2}\right)^\trans-\left(\frac{G_2\ones}{n+1}+ \frac{\xfat}{2}\right)\ones^\trans+\left(\frac{\ones^\trans G_2\ones}{(n+1)^2}+ \frac{1}{4}\right)\ones\ones^\trans\\
&=G_2- \ones\frac{\ones^\trans G_2}{n+1}-\frac{G_2\ones}{n+1}\ones^\trans+\frac{\ones^\trans G_2\ones}{(n+1)^2}\ones\ones^\trans+\xfat\xfat^\trans-\ones\frac{\xfat^\trans}{2}-\frac{\xfat}{2}\ones^\trans+ \frac{1}{4}\ones\ones^\trans\\
&=\left(I-\frac{\ones\ones^\trans}{n+1}\right)G_2\left(I-\frac{\ones\ones^\trans}{n+1}\right)+\left(\xfat- \frac{\ones}{2}\right)\left(\xfat- \frac{\ones}{2}\right)^\trans,
\end{align*}
which we recognize from \eqref{MPinvofA}.
\end{proof}

We remark that the explicit expression of $G_2$ was derived in \cite{InversesEriksson2020preprint} for second and fourth order accurate SBP operators.

\subsection{Pseudoinverse with filtering}

\newcommand{\tmpvec}{\mathbf{z}}
\newcommand{\arbvec}{\mathbf{y}}
\newcommand{\vB}{\mathbf{v}_B}


The solution $\discsol_{\text{\tiny MP}}\equiv A^+\rhs$ produced by the Moore--Penrose inverse $A^+$ has meanvalue equal to zero (i.e. the property $\ones^\trans\discsol_{\text{\tiny MP}}=0$), since $A^+A=I-\frac{1}{n+1}\ones\ones^\trans$.
Below, we present an alternative approach  by first deriving a pseudoinverse and then using a filtering process to obtain the same solution $\discsol_{\text{\tiny MP}}$.

Let $B$ be a candidate for acting as a pseudo-inverse. We multiply the system \eqref{Avb} by this matrix $B$ from the left, yielding
$BA\discsol=B\rhs$. We demand that $BA\discsol=\discsol+\ones \tmpvec^\trans\discsol$ where $\tmpvec$ is an unknown vector to be determined. If this  demand is fulfilled,  $\vB=B\rhs$ will be equal to the Moore-Penrose solution plus some constant given by $\tmpvec^\trans\discsol$.

We thus need $B$ such that $BA=I+\ones\tmpvec^\trans$. Making the ansatz $B=G_2+C$, with $G_2$ from \eqref{G2}, yields
\begin{align*}
BA&=(G_2+C)A\\
&=I-(\ones-\xfat )\el^\trans-\xfat\er^\trans +CA\\
&=I-\ones \el^\trans-\xfat(\er-\el)^\trans +CA
\end{align*}
where we have used \eqref{AG} and that both $A$ and $G_2$ are symmetric. Based on \eqref{A1Ax} we let $C=\xfat\xfat^\trans+\ones\arbvec^\trans$, where $\arbvec$ is arbitrary, since this leads to
\begin{align*}
BA&=I+\ones(\arbvec^\trans A- \el^\trans),
\end{align*}
fulfilling our initial requirement with $\tmpvec=A\arbvec-\el$. 
That is, we have
\begin{align}\label{altpsinv}
B=G_2+\xfat\xfat^\trans+\ones\arbvec^\trans.
\end{align}



Choosing $\arbvec=0$ minimizes the computations needed to obtain $B$, and thereby $\vB$.
Thereafter subtracting the average of the solution from itself  gives the Moore-Penrose solution: 
\begin{align*}
 \vB-\frac{1}{\ones^\trans\ones}\ones\ones^\trans \vB.
\end{align*}
Comparing with using the analytical expression of $A^+$, the above approach may be more computational efficient when the right hand side vector $\rhs$ has compact support, because in this case only a small part of $G_2$ needs to be formed.







\section{The relation between $\operatorname{rank}(\Abar)$ and $\operatorname{rank}(A)$}
\label{ranks}

In order for Theorem~\ref{ThmMP} to make any sense, it is necessary that $\Abar$ in $G_2$ in \eqref{G2} is non-singular. In \cite{InversesEriksson2020preprint}, this is proven directly for the second and fourth order accurate operators.
However, in Section~\ref{SecGenA} we show that there is one free parameter in the sixth order accurate operator, and that it is possible to tune this free parameter such that all SBP properties are satisfied but $\Abar$ is singular. Below, we first present a theorem relating the rank of $\Abar$ and $A$. 

\begin{theorem}
\label{ThmRankRelationA}
Consider the $(n+1)\times(n+1)$ symmetric matrix $A$ in \eqref{SBP_D2}, which fulfills the consistency constraints \eqref{A1Ax}. Further, consider its $(n-1)\times(n-1)$ submatrix $\Abar$ in \eqref{Aparts}. It holds that
\begin{align*}
\operatorname{rank}(A)=    \operatorname{rank}(\Abar)+1.
\end{align*}
\end{theorem}

\newcommand{\Sylv}{Z}
\newcommand{\RotA}{\Delta}

\begin{proof}[Proof of Theorem~\ref{ThmRankRelationA}]

We multiply $A$ in \eqref{Aparts} by $\Sylv$ from the left and by $\Sylv^\trans$ from the right, where the non-singular matrix $\Sylv$ is specified below. This gives us a resulting matrix $\RotA$, as
\begin{align}\label{Bsylv}
\RotA=\Sylv A\Sylv^\trans&=\left[\begin{array}{ccc}1&\ett^\trans&1\\0&\Ibar&0\\0&\xvec^\trans&1\end{array}\right]\left[\begin{array}{ccc}\Abeg&\Awest^\trans&\Aoff\\\Awest&\Abar&\Aeast\\\Aoff&\Aeast^\trans&\Aend\end{array}\right]\left[\begin{array}{ccc}1&0&0\\\ett&\Ibar&\xvec\\1&0&1\end{array}\right]
=\left[\begin{array}{ccc}0&0&0\\0&\Abar&0\\0&0&1\end{array}\right]
\end{align}
where $\ett=[1\ 1\ 1\ \hdots\ 1]^\trans$ and $\xvec=[x_1\ x_2\ \hdots\ x_{n-1}]^\trans$ are shorter versions of $\ones$ and $\xfat$ in \eqref{onesandxfat}, and where $\Ibar$ is the $(n-1)\times(n-1)$ identity matrix.
 To arrive to the right-hand expression, we have used the relations in \eqref{A1Axcomponents}.
Since the rank of $A$ does not change under the above transformation, $\operatorname{rank}(\RotA)=\operatorname{rank}(A)$ must hold (due to Sylvester's law of inertia). On the other hand, $\RotA$ computed in \eqref{Bsylv} is a block diagonal matrix with entries $\Abar$ and $1$, and consequently $\operatorname{rank}(\RotA)=\operatorname{rank}(\Abar)+1$. We conclude that $\operatorname{rank}(A)=\operatorname{rank}(\Abar)+1$.

\end{proof}

\begin{corollary}\label{CorRank}
In particular, from  Theorem~\ref{ThmRankRelationA} it directly follows that
\begin{align*}
  \operatorname{rank}(A)=n
  &&\Longleftrightarrow&&
  \operatorname{rank}(\Abar)=n-1.
\end{align*}
That is, $A$ has exactly one zero eigenvalue if and only if $\Abar$ is invertible.
\end{corollary}
For the existence of $G_2$ in Theorem~\ref{ThmMP}, $\Abar$ needs to be non-singular and
from Corollary~\ref{CorRank}, we know that $\Abar$ 
is non-singular if and only if $A$ has exactly one zero eigenvalue. This rises 
the question whether or not $A$ can have more than one zero eigenvalue and still fulfill \eqref{SBP_D2}.

Recall that for the second and fourth order accurate operators, it was proven in \cite{InversesEriksson2020preprint} that the related $\Abar$ is non-singular, which together with Corollary~\ref{CorRank} proves that $\operatorname{rank}(A) = n$. This is often assumed to be the case, but to the best of our knowledge this has not been shown before. We summarize this result in the following corollary.
\begin{corollary}\label{CorRank24}
For the second and fourth order second derivative SBP operators constructed in \cite{Mattsson2004503}, the corresponding matrix $A$ has exactly one zero eigenvalue.
\end{corollary}
However, for higher order accurate operators we will demonstrate below that it is actually possible to construct $A$ from \eqref{SBP_D2} such that it has more than one zero eigenvalue. 

\section{The sixth order second derivative operator}
\label{SecGenA}

As already mentioned, the inverse of $\Abar$ was derived in \cite{InversesEriksson2020preprint} for operators of second and fourth order of accuracy, thereby guaranteeing that $\Abar$ has full rank and consequently that $\operatorname{rank}(A)=n$. Higher order operators, on the other hand, have free parameters. Is it 
possible to tune these parameters such that $A$ has more than one zero eigenvalue?
We will investigate the sixth order accurate operator, which turns out to have one free parameter.

\subsection{Construction of the sixth order operator $D_2$ with a free parameter}
\label{constructA}

\newcommand{\free}{\alpha}


\newcommand{\aaabas}{c}
\newcommand{\ccbas}{d}
\newcommand{\aaa}{\aaabas_{00}}%
\newcommand{\aab}{\aaabas_{01}}%
\newcommand{\aac}{\aaabas_{02}}%
\newcommand{\aad}{\aaabas_{03}}%
\newcommand{\aae}{\aaabas_{04}}%
\newcommand{\aaf}{\aaabas_{05}}%
\newcommand{\aba}{\aaabas_{10}}%
\newcommand{\abb}{\aaabas_{11}}%
\newcommand{\abc}{\aaabas_{12}}%
\newcommand{\abd}{\aaabas_{13}}%
\newcommand{\abe}{\aaabas_{14}}%
\newcommand{\abf}{\aaabas_{15}}%
\newcommand{\aca}{\aaabas_{20}}%
\newcommand{\acb}{\aaabas_{21}}%
\newcommand{\acc}{\aaabas_{22}}%
\newcommand{\acd}{\aaabas_{23}}%
\newcommand{\ace}{\aaabas_{24}}%
\newcommand{\acf}{\aaabas_{25}}%
\newcommand{\ada}{\aaabas_{30}}%
\newcommand{\adb}{\aaabas_{31}}%
\newcommand{\adc}{\aaabas_{32}}%
\newcommand{\add}{\aaabas_{33}}%
\newcommand{\ade}{\aaabas_{34}}%
\newcommand{\adf}{\aaabas_{35}}%
\newcommand{\aea}{\aaabas_{40}}%
\newcommand{\aeb}{\aaabas_{41}}%
\newcommand{\aec}{\aaabas_{42}}%
\newcommand{\aed}{\aaabas_{43}}%
\newcommand{\aee}{\aaabas_{44}}%
\newcommand{\aef}{\aaabas_{45}}%
\newcommand{\afa}{\aaabas_{50}}%
\newcommand{\afb}{\aaabas_{51}}%
\newcommand{\afc}{\aaabas_{52}}%
\newcommand{\afd}{\aaabas_{53}}%
\newcommand{\afe}{\aaabas_{54}}%
\newcommand{\aff}{\aaabas_{55}}%
\newcommand{\ca}{\ccbas_{0}}%
\newcommand{\cb}{\ccbas_{1}}%
\newcommand{\cc}{\ccbas_{2}}%
\newcommand{\cd}{\ccbas_{3}}%


The narrow-stencil SBP operators constructed in 
 \cite{Mattsson2004503} are designed
such that the following requirements 
are fulfilled:
\begin{subequations}\label{DemandsOnA}
\begin{align}
\label{DemandsOnA(i)}
\bullet\ &\text{The diagonal matrices $H$ in  Definitions~\ref{defD1} and \ref{defD2} are identical.}\\
\label{DemandsOnA(ii)}
\bullet\ &\text{$D_2$ (and thus $A$) has minimal interior bandwidth (the interior stencil has $2p+1$ elements)}.\\
\label{DemandsOnA(iii)}
\bullet\ &\text{The order of accuracy of $D_2$ is $2p$ in the interior and $p$ near the boundaries.}\\
\label{DemandsOnA(iv)}
\bullet\ &\text{The order of accuracy of $\dselr$ is $p+1$.}\\
\label{DemandsOnA(v&vi)}
\bullet\ &\text{$A$ is symmetric positive semi-definite, that is $A=A^\trans\geq0$.}
\end{align}
\end{subequations}
%
Below, we derive the sixth order diagonal-norm $D_2$ fulfilling \eqref{DemandsOnA}, revealing that it has a free parameter that is not presented in \cite{Mattsson2004503}. The existence of the free parameter was mentioned in \cite{ Sjogreen2012} but no such operator was reported.


In the sixth order accurate case, the ansatz for the top-left corner of $A$ is
\begin{align}\label{ansatzA}
A=\frac{1}{h}\left[\begin{array}{cccccccccccc}
\aaa&\aab&\aac&\aad&\aae&\aaf\\
\aba&\abb&\abc&\abd&\abe&\abf\\
\aca&\acb&\acc&\acd&\ace&\acf\\
\ada&\adb&\adc&\add&\ade&\adf&\cd\\
\aea&\aeb&\aec&\aed&\aee&\aef&\cc&\cd\\
\afa&\afb&\afc&\afd&\afe&\aff&\cb&\cc&\cd\\
&&&\cd&\cc&\cb&\ca&\cb&\cc&\cd\\
&&&&\ddots&\ddots&\ddots&\ddots&\ddots&\ddots&\ddots
\end{array}\right],
\end{align}
and a corresponding ansatz is also made for the bottom-right corner.
The width of the interior stencil is set to 7 elements by \eqref{DemandsOnA(ii)}, and the coefficients $\ca$, $\cb$, $\cc$ and $\cd$ are known from requiring 
sixth order accuracy in the interior according to \eqref{DemandsOnA(iii)}, 
yielding the interior stencil 
\begin{align*}
\left[\begin{array}{ccccccc}\cd&\cc&\cb&\ca&\cb&\cc&\cd\\\end{array}\right]=\frac{1}{180}\left[\begin{array}{ccccccc}-2&27&-270&490&-270&27&-2\\\end{array}\right].
\end{align*}
Demanding symmetry, \eqref{DemandsOnA(v&vi)}, 
gives us $\aaabas_{ij}=\aaabas_{ji}$ and reduce the number of unknowns $\aaabas_{ij}$ to 21 for $0\leq i\leq j\leq5$.

Next, we require accuracy at the boundary of $D_2$ and $\dselr$, that is \eqref{DemandsOnA(iii)} and \eqref{DemandsOnA(iv)}. Let $\smooth(x)$ denote an arbitrary smooth function in $\Omega$, and let $\restrict=[\smooth(x_0)\ \smooth(x_1)\ \hdots\ \smooth(x_{n})]^\trans$ 
be its restriction to the grid. 
Similarly, let $\restrict^{(m)}$  denote the restriction of the $m$th 
derivative $\smooth^{(m)}$ 
to the grid. 
At the boundary, we thus need
\begin{align*}
D_2\restrict=\restrict^{(2)}+\mathcal{O}(h^3),&&\dsel^\trans\restrict = \smooth_x(0)+\mathcal{O}(h^4) ,&&\dser^\trans\restrict = \smooth_x(1)+\mathcal{O}(h^4).
\end{align*}
We now need to choose the coefficients $\aaabas_{ij}$ in $A$ such that the above relations are fulfilled.
From Definition~\ref{defD2}, we get $A\restrict=-H D_2\restrict+\er\dser^\trans\restrict-\el\dsel^\trans\restrict$, that is, $A$ and $H$ must be such that
\begin{align}\label{demandsonA}
A\restrict
=-H \restrict^{(2)}+\er\smooth_x(1)-\el \smooth_x(0)+\mathcal{O}(h^4)
\end{align}
holds (recall that the elements of $H$ are proportional to $h$).
We require that the first six rows of 
\eqref{demandsonA} hold for any smooth function $w(x)$, 
and consider
the  Taylor expansions of  $(A\restrict)_j$ around $\smooth(x_j)$ for $j=0,1,2,3,4,5$. As an example, the first row of \eqref{demandsonA} leads to the following four conditions:
\begin{align*}
\aaa+\aab+\aac+\aad+\aae+\aaf&=0,&\aab+2\aac+3\aad+4\aae+5\aaf&=-1,\\
\aab+8\aac+27\aad+64\aae+125\aaf&=0,&\aab+16\aac+81\aad+256\aae+625\aaf&=0,
\end{align*}
corresponding to the first, second, fourth and fifth term in the Taylor expansion. Note that we have not yet put any demands on the third term of the Taylor expansion, which involves $\restrict^{(2)}$ and thus the SBP norm $H$. This   will be checked later (by checking if $h_0=-(\aab+4\aac+9\aad+16\aae+25\aaf)\frac{h}{2}$, where $h_0$ is 
the first diagonal element of $H$, coincide with the norm for the first derivative). 

All in all, requiring  the above accuracy demands, puts 4 demands on each of the top six rows of $A$, that is 24 demands in total. Combined with the symmetry requirement, these demands can be written as a linear system of equations with 24 equations and 21 unknowns $\aaabas_{ij}$. For completeness, this system is presented explicitly in Appendix~\ref{Syst2421}.
However, these equations turn out to be linearly dependent, four of them being superfluous. The remaining 20 demands are without conflicts, and result in the top-left corner of $A$ being
\begin{align}\label{A6}\begin{split}
A_{6\times6}
&=\frac{1}{180h}\scalebox{.93}{$\left[\begin{array}{cccccc}\frac{-19697}{72}&\frac{2098907}{960}&\frac{-3475609}{720}&\frac{6987397}{1440}&\frac{-193649}{80}&\frac{278033}{576}\\\frac{2098907}{960}&\frac{-839647}{72}&\frac{6921397}{288}&\frac{-387859}{16}&\frac{6969449}{576}&\frac{-1739359}{720}\\ \frac{-3475609}{720} & \frac{6921397}{288} & \frac{-577009}{12} & \frac{6943085}{144} & \frac{-3481031}{144} & \frac{2321591}{480}\\ \frac{6987397}{1440} & \frac{-387859}{16} & \frac{6943085}{144} & \frac{-1726033}{36}& \frac{2298631}{96} & \frac{-3473101}{720}\\\frac{ -193649}{80} & \frac{6969449}{576}& \frac{-3481031}{144} & \frac{2298631}{96}& \frac{-104756}{9} & \frac{6235729}{2880}\\ \frac{278033}{576} & \frac{-1739359}{720} & \frac{2321591}{480} & \frac{-3473101}{720} & \frac{6235729}{2880} & 0 \end{array}\right]$}+\frac{\free}{180h}\scalebox{.93}{$\left[\hspace{-2pt}\begin{array}{c}1\\-5\\10\\-10\\5\\-1\end{array}\hspace{-2pt}\right]\left[\hspace{-2pt}\begin{array}{c}1\\-5\\10\\-10\\5\\-1\end{array}\hspace{-2pt}\right]^\trans$}
\end{split}
\end{align}
where $\free$ is a free parameter.

We also need to control the norm.
We compute the first six diagonal elements of $H$
\begin{align*}
\left[\begin{array}{l}
h_{0}\\h_{1}\\h_{2}\\h_{3}\\h_{4}\\h_{5}\end{array}\right]=-\frac{h}{2}\left[\begin{array}{l}
\aab+4\aac+9\aad+16\aae+25\aaf\\
\aba+\abc+4\abd+9\abe+16\abf\\
4\aca+\acb+\acd+4\ace+9\acf\\
9\ada+4\adb+\adc+\ade+4\adf+9\cd\\
16\aea+9\aeb+4\aec+\aed+\aef+4\cc+9\cd\\
25\afa+16\afb+9\afc+4\afd+\afe+\cb+4\cc+9\cd\end{array}\right]=
\frac{h}{43200}
\left[\begin{array}{l}
13649\\60065\\27110\\53590\\39385\\43801\end{array}\right],
\end{align*}
where $H$ turns out to be independent of the free parameter and -- more importantly -- comparing with \cite{ref:STRA94} we see that it is equal to the norm used for the first derivative. The requirement \eqref{DemandsOnA(i)} is thus fulfilled automatically.

With \eqref{A6} and its bottom-right corner counterpart inserted into \eqref{ansatzA}, we have constructed a matrix $A$ fulfiling all demands 
\eqref{DemandsOnA(i)}, \eqref{DemandsOnA(ii)}, \eqref{DemandsOnA(iii)}, \eqref{DemandsOnA(iv)} and 
symmetry from  \eqref{DemandsOnA(v&vi)}.
It remains to make sure $A\geq0$, that is fulfilling the positive semi-definiteness  demand from \eqref{DemandsOnA(v&vi)}. 
Note that in \cite{Mattsson2004503}, no free parameter is presented, with their $A$ obtained by choosing $\free=490$. 
In this case $A\geq0$, 
as a consequence of Proposition 3.5 in \cite{Mattsson2008b} (since $R\geq0$ for $R$ in Definition~\ref{DefComp} implies $A\geq0$).
Furthermore, empirical evidence shows that this particular matrix $A$ has one and only one zero eigenvalue. We formalize this in the following assumption: 
\begin{assumption}\label{rankAn}
For $\free=490$, $A$ is positive semi-definite, with $\operatorname{rank}(A)=n$.
\end{assumption}

In the following subsection, we investigate for what values of $\free$ it holds $A\geq 0$. As will be seen, it is possible that $A\geq 0$ has more than one zero eigenvalue.

\subsection{Rank of $A$}

\newcommand{\Abase}{A_0}
\newcommand{\extra}{K}
\newcommand{\extraev}{\mathbf{k}}

Based on \eqref{ansatzA} with \eqref{A6} -- and under the assumption that both boundary closures are equivalent -- we let $A=\Abase+\free\extra$ where $\Abase=A(\free=0)$. The matrix $\extra$ consists entirely of zeros, except its upper left $6\times6$ corner, $\extra_{6\times6}$ seen multiplied by $\free$ in \eqref{A6},
 and its lower right $6\times6$ corner.
 This matrix $\extra$, factorized as 
\begin{align*}
\extra=\frac{1}{180h}[\extraev_1 \extraev_2][\extraev_1 \extraev_2]^\trans,&&
\begin{array}{l}\extraev_1=\left[\begin{array}{ccccccccccccccc}1&-5&10&-10&5&-1&0&\hdots&0&0&0&0&0&0&0\end{array}\right]^\trans,\\
\extraev_2=\left[\begin{array}{ccccccccccccccc}
0&0&0&0&0&0&0&\hdots&0&-1&5&-10&10&-5&1\end{array}\right]^\trans,\end{array}
\end{align*}
is positive semi-definite with two eigenvalues  having the value $\frac{252}{180h}=\frac{1.4}{h}$ (associated with the eigenvectors $\extraev_{1,2}$), and with the rest of the eigenvalues 
being zero.
Note that $A=\free(\extra+\Abase/\free)$, and for sufficiently large values of $|\free|$ the eigenvalues of $A$ will be perturbed versions of the ones of $\free\extra$. 
That is,  we expect $A$ to have two eigenvalues scaling roughly as $\frac{1.4}{h}\free$ as $|\free|\to\infty$. Thus $A$ cannot be positive semi-definite for large negative values of $\free$. However, we know that  $A$  is positive semi-definite for $\free=490$. Since these two eigenvalues are continuous functions of $\free$ (they are the roots of the  characteristic polynomial of $A$, whose coefficients in turn depends continuously on $\free$), they will take the value zero for some value of $\free\in(-\infty,490)$. 

We investigate the precise values of $\free$ when $A$ has an extra zero eigenvalue (in addition to the expected zero eigenvalue associated with the eigenvector $\ones$). Since it is easier to investigate if a matrix is singular than to see if it has more than one zero eigenvalue, we will benefit from the relation between $A$ and $\Abar$. For symmetric matrices, the rank is equal to the number of nonzero eigenvalues. Thus, according to Corollary~\ref{CorRank}, $A$ has only one zero eigenvalue if $\Abar$ is non-singular. Equivalently, $A$ has more than one zero eigenvalue if $\Abar$ is singular.  The following proposition gives the precise limiting value of $\free$ that leads to a singular $\Abar$, thus more than one zero eigenvalue of $A$. 

\newcommand{\limit}{\free^\star}
\begin{proposition}\label{propfreelimit}
The matrix $A$ in Definition~\ref{defD2}, with upper left corner as indicated in \eqref{A6}, is positive semi-definite for $\free\geq\limit$, where
\begin{align}\label{freelimit}
 \limit\approx481.3408873321106,
\end{align}
for $n\geq21$ (for $n<21$, $\limit$ is slightly larger, with maximum $\limit\approx481.3410851822219$ for $n=11$).
Moreover, for $\free=\limit$, the matrix $A$ has two zero eigenvalues and one eigenvalue that is very close to zero.
\end{proposition}

\newcommand{\Abarbase}{\Abar_0}
\newcommand{\extrabar}{\bar{\extra}}


\newcommand{\Abarorg}{\Abar_{490}}
\newcommand{\delextra}{E}
\newcommand{\delextrabar}{\bar{E}}

\begin{proof}[Proof of Proposition~\ref{propfreelimit}]

\newcommand{\Ibarproof}{I_{n-1}}

We start by letting $\Abar=\Abarbase+\free\extrabar$ be the interior part of $A$, as indicated in \eqref{Aparts}. In the same sense $\extrabar$ is the interior part of $\extra$, and can be factorized as
\begin{align*}
\extrabar=\frac{1}{180h}\delextrabar\delextrabar^\trans,&&
\delextrabar=\left[\begin{array}{ccccccccccccc}-5&10&-10&5&-1&0&\hdots&0&0&0&0&0&0\\
0&0&0&0&0&0&\hdots&0&-1&5&-10&10&-5\end{array}\right]^\trans
\end{align*}
where $\delextrabar$  is an $(n-1)\times2$-matrix.
Next, let  $\Abarorg=\Abarbase+490\extrabar$ be the particular choice of $\Abar$ correlated to operator $A$ derived in \cite{Mattsson2004503}.
Combining the mentioned expressions leads to $\Abar=\Abarorg+(\free-490)\extrabar$.
From Assumption~\ref{rankAn} together with Corollary~\ref{CorRank}, we know that $\Abarorg$ is positive definite.

We compute the determinant of $\Abar$ to find out for what values of $\free$ it becomes zero. We have
\begin{align}\label{detAbar}\begin{aligned}
\det(\Abar)&=\det(\Abarorg+(\free-490)\extrabar)\\
&=\det(\Abarorg)\det\left(\Ibarproof+\frac{(\free-490)}{180h}\Abarorg^{-1}\delextrabar\delextrabar^\trans\right)\\
&=\det(\Abarorg)\det\left(I_2+\frac{(\free-490)}{180h}\delextrabar^\trans\Abarorg^{-1}\delextrabar\right)
\end{aligned}
\end{align}
where we have used the following two determinant properties
\begin{align*}
\det(MN) &= \det(M)\det(N),&& M, N  \text{ are square matrices of the same size}
\\
\det(I_{k}+MN^\trans)&=\det(I_{l}+N^\trans M),&&I_j \text{ is the }j\times j \text{ identity matrix}, M\text{ and }N\text{ are }k\times l\text{ matrices,}
\end{align*}
where the latter is the Weinstein–Aronszajn identity.
In \eqref{detAbar}, we have  used that the matrix $\Abarorg$ is invertible. 

We note that $\Abar$ is singular only if the determinant of the $2\times2$-matrix $I_2+\frac{(\free-490)}{180h}\delextrabar^\trans\Abarorg^{-1}\delextrabar$ in \eqref{detAbar} is zero.  We define $\limit$ as the value of $\free$ for which this occur, and 
compute 
\begin{align*}
0&=\det\left(I_2+\frac{(\limit-490)}{180h}\delextrabar^\trans\Abarorg^{-1}\delextrabar\right)\\
&=\left(\frac{(\limit-490)}{180h}\right)^2\det\left(\delextrabar^\trans\Abarorg^{-1}\delextrabar+\frac{180h}{(\limit-490)} I_2\right)\\
&=\left(\frac{(\limit-490)}{180h}\right)^2\det(\delextrabar^\trans\Abarorg^{-1}\delextrabar-\widetilde{\lambda} I_2)
\end{align*}
where $\widetilde{\lambda}=-\frac{180h}{(\limit-490)}$ are the eigenvalues of the $2\times2$-matrix $\delextrabar^\trans\Abarorg^{-1}\delextrabar$. To know when $\Abar$ is singular, we thus solve for $\limit$ in terms of $\widetilde{\lambda}$. We obtain that when
\begin{align*}
\limit=490-\frac{180h}{\widetilde{\lambda}},
\end{align*}
$\Abar$ is singular.
%
Since $\delextrabar^\trans\Abarorg^{-1}\delextrabar$ has two eigenvalues (which we compute numerically), there are two solutions $\limit$. These values vary slightly as a function of the number of grid points $n$, as shown in Table~\ref{Tab:ConvFree} in Appendix~\ref{Syst2421}.
From now on, we re-define $\limit$ such that it refers to the maximum of these two values. Then $\limit$ is the value when the determinant of $\Abar$ first changes sign, and is thus the smallest value of $\free$ such that $\Abar$ is positive definite (that is $\Abar$ is positive definite for $\free>\limit$ and positive semi-definite for $\free=\limit$).
For larger values of $n$ this distinction is irrelevant, since we note that both values of $\free$ converges to $481.3408873321106$ and are indistinguishable  in double precision for $n\geq21$. 

Theoretically, $\Abar$ is singular and positive semi-definite for $\free=\limit$, with one eigenvalue equal to zero and one eigenvalue very close to zero. However, in practice, $\Abar$ with $\free=\limit$ has two zero eigenvalues when $n\geq21$ (the second smallest eigenvalue of $\Abar$ with $\free=\limit$ decreases rapidly when $n$ increases).
Recalling Theorem~\ref{ThmRankRelationA} we now conclude that $A\geq0$ for $\free\geq\limit$.


\end{proof}

 A consequence of  Proposition~\ref{propfreelimit} is that, for $\free=\limit$, the matrix $A$ fulfills the properties of Definition~\ref{defD2} but fails to have $\operatorname{rank}(A)=n$. Instead, $A$ has rank $n-1$ (in practice, that is in numerical computations, the rank is rather $n-2$ for $n\geq21$).


As a numerical verification of the limit in \eqref{freelimit} and of Theorem~\ref{ThmRankRelationA}, we consider the eigenvalues of $A$, and see how they are influenced by the choice of $\free$.
In Figure~\ref{LargestSmallestEigOfA20}, the eigenvalues of $A$ are shown as a function of the free parameter $\free$ for $n=24$.
\begin{figure}[htb]\centering
\includegraphics[width=.6\textwidth]{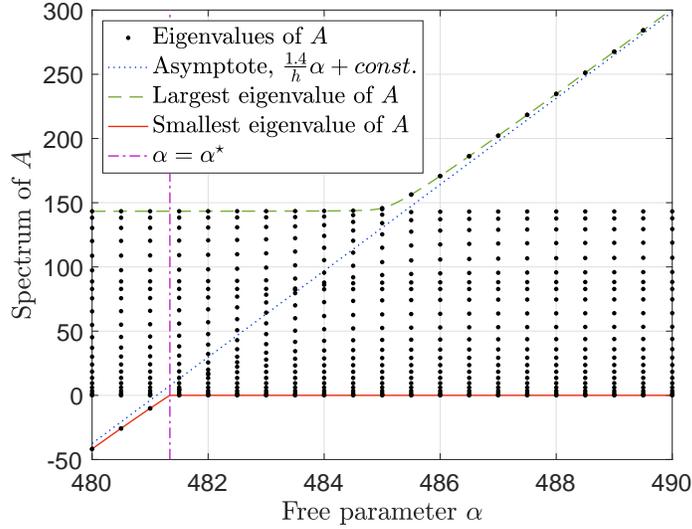}
\caption{The spectrum of $A$, as a function of $\free$ (for $n=24$).}
\label{LargestSmallestEigOfA20}
\end{figure}   
As expected, we need $\free\geq\limit$ from \eqref{freelimit} to obtain $A\geq0$.
Note that in order to keep the largest eigenvalue of $A$ small, one should not exceed $\free\approx484.9$, which indicates that  the original value $\free=490$ is not the optimum choice.  
 It is crucial to consider how $\free$ affects the properties of $D_2$ and, even more importantly, $\DDN$.
This will be investigated in detail  in Section~\ref{InfluenceOfFree}.

\subsection{The borrowing capacity}

\newcommand{\indexc}{\text{\tiny C}}
\newcommand{\qhatc}{\xi_\indexc}
\newcommand{\qhatl}{\xi_\indexl}
\newcommand{\qhatr}{\xi_\indexr}
\newcommand{\qhatlr}{\xi_\indexlr}

\newcommand{\indexrl}{\text{\tiny R,L}}
\newcommand{\dserl}{\mathbf{d}_\indexrl}

In the semi-discretization of the wave equation with Dirichlet boundary conditions \eqref{waveD_semi}, we need the borrowing capacity $\app$. Since $\app$ depends on $A$, its known numerical value is only valid for the standard choice $\free=490$ from \cite{Mattsson2004503}. When $\free$ is changed, $\app$ needs to be adjusted.
\begin{figure}[H]
    \centering
    \includegraphics[width=0.6\textwidth]{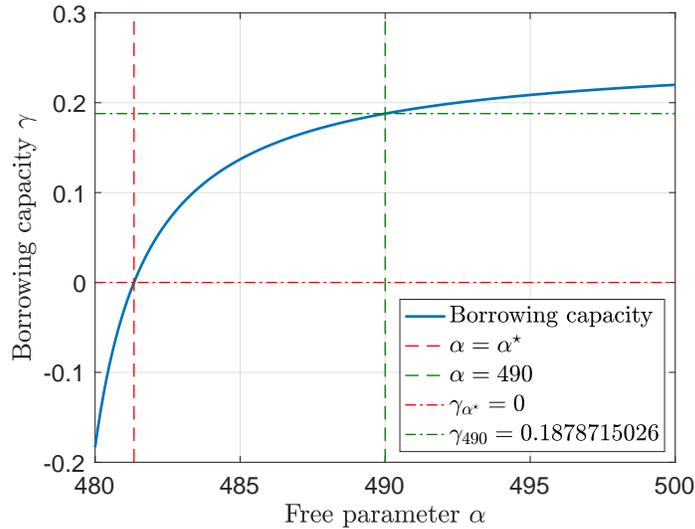}
    \caption{The borrowing capacity 
    $\app$ as a function of the free parameter $\free$ in $A$.
    }
    \label{fig:AnewBorrow2}
\end{figure}

According to Theorem 3.5 in \cite{InversesEriksson2020preprint}, $\app$ can be computed using
\begin{align}\label{ComputeApp}
\app=\frac{1}{h(\qhatlr+|\qhatc|)},
&&\qhatlr=1+\dselr^\trans G_2\dselr,&& \qhatc=1+\dserl^\trans G_2\dselr.
\end{align}
The resulting $\app$ is shown in Figure~\ref{fig:AnewBorrow2}. The value varies slightly with $n$, but for $n\geq21$ it has converged numerically  (in double precision). The number $\gamma_{490}=0.1878715026$ shown in Figure~\ref{fig:AnewBorrow2} is the known value computed for the standard choice $\free=490$  in \cite{MATTSSON20088753}.
Note that if the free parameter is chosen as $\free=\limit$ from \eqref{freelimit}, then there is no extra "positivity" in $A$ available, and it is not possible to use the "borrowing technique".


\subsection{Compatibility of $D_1$ and $D_2$}
\label{SecComp}

\newcommand{\freeD}{\beta}

As mentioned, when solving PDEs with both the first and second derivatives, it is important that the SBP operators $D_1$ and $D_2$ are compatible.  
The SBP operators $D_1$ and $D_2$ are compatible if 
$R = A - D_1^\trans H D_1$ in Definition~\ref{DefComp}
is symmetric positive semi-definite.
The standard sixth order accurate $D_2$ with $\free=490$ is compatible with $D_1$ constructed in  \cite{ref:KREI74,ref:STRA94}.
Now we are interested if this holds also for other choices of $\free$.

Note that the sixth order accurate $D_1$  also has one free parameter, $\freeD$ (denoted $x_1$ in \cite{ref:STRA94}). Whether or not $D_1$ and $D_2$ are compatible depend not only on $\free$ but also on $\freeD$, as we see in Figure~\ref{fig:CompatibilityRegion2}.
\begin{figure}[hbt]\centering
\subfigure{\includegraphics[width=.5\textwidth]{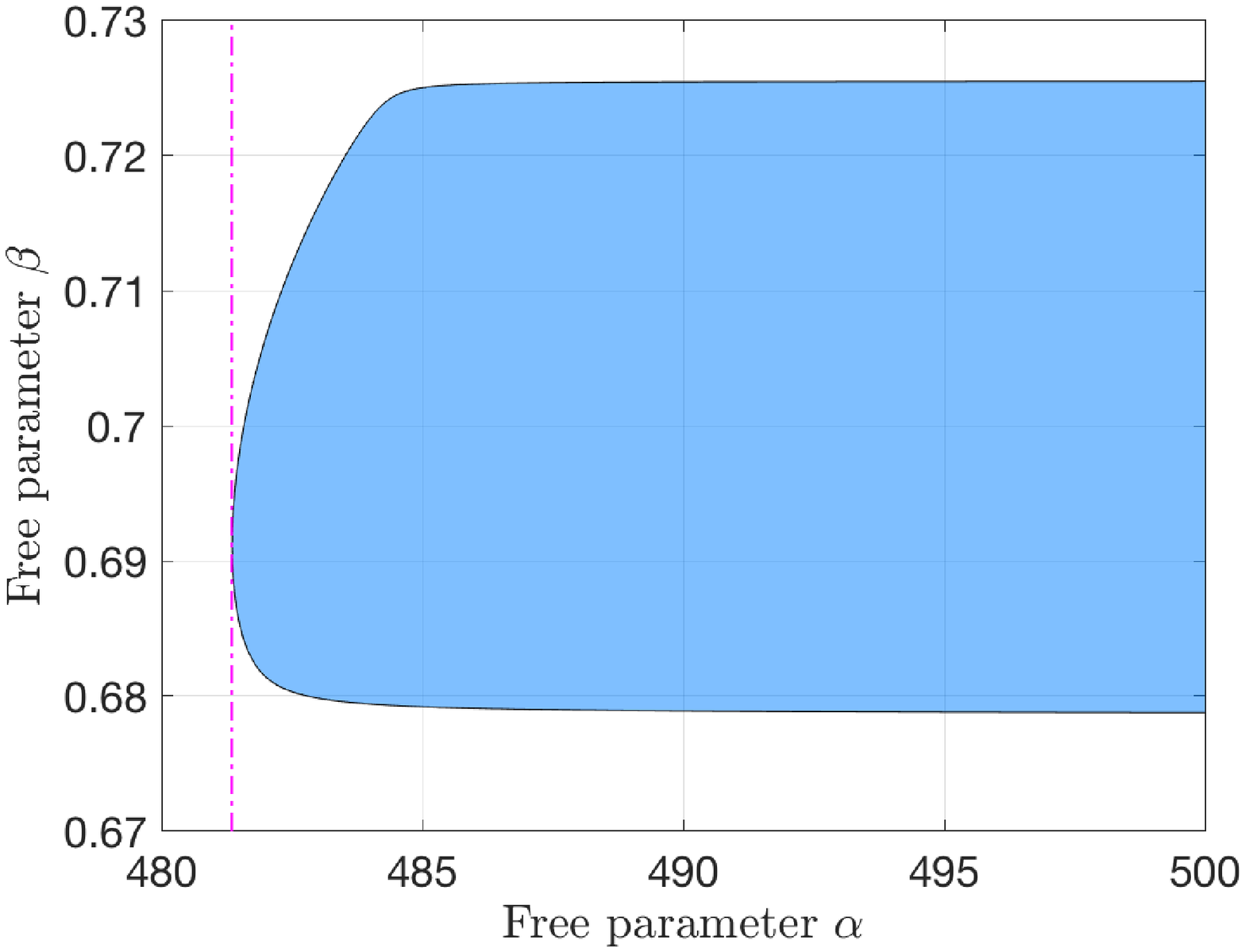}}%
\subfigure{\includegraphics[width=.5\textwidth]{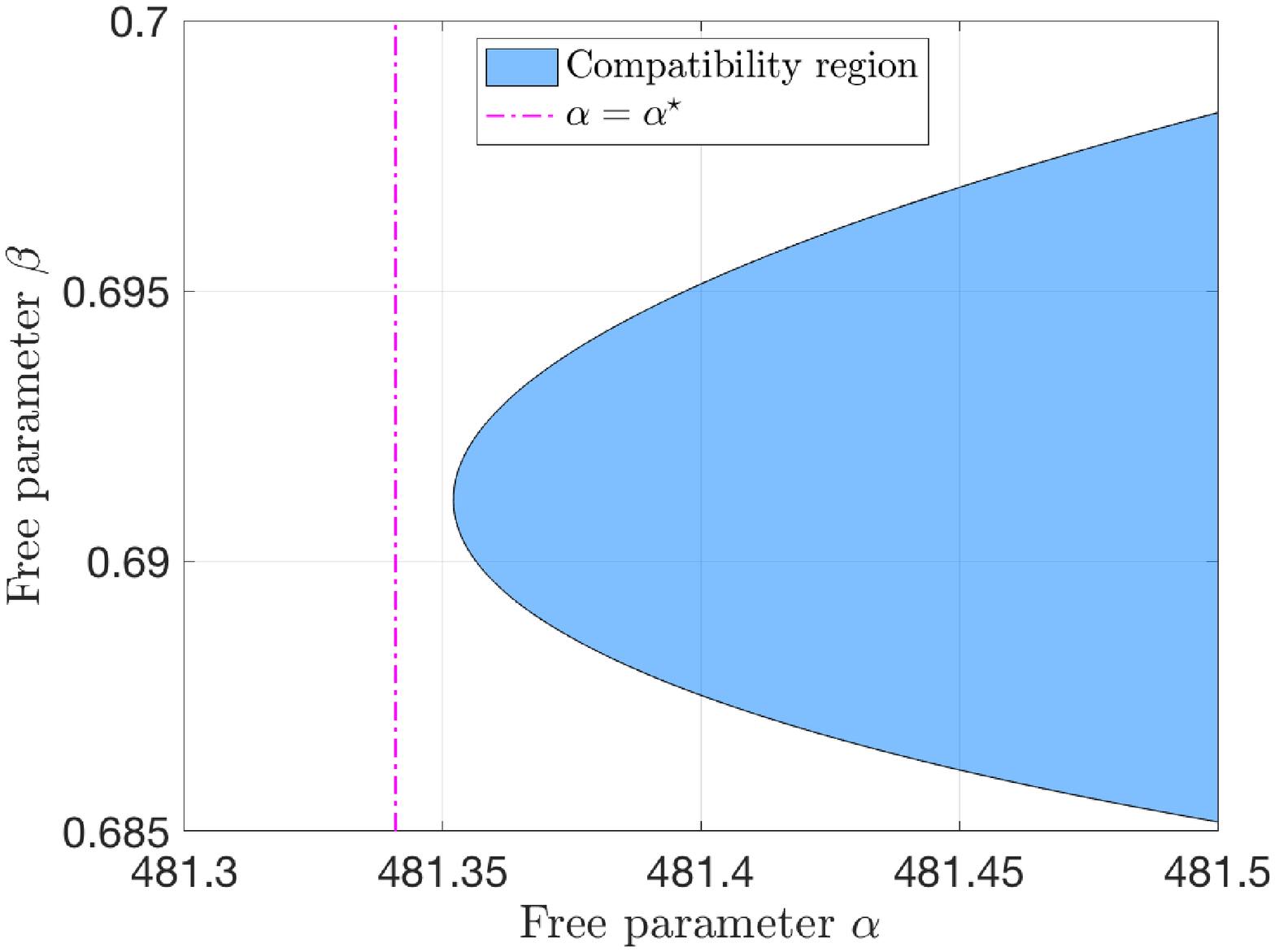}}
\caption{Compatibility region of $D_1$ and $D_2$. If zooming in very much, one finds that it is necessary with $\free>481.35207212433>\limit$ for compatibility (with $\freeD=0.69113483$).}
\label{fig:CompatibilityRegion2}
\end{figure} 
Naturally, since the eigenvalues of $A$ become increasingly positive as $\free$ increases, so does the range of compatible values of $\freeD$. Since $R$ is less positive than $A$, the minimum accepted value $\free=\limit$ is not enough to make $D_1$ and $D_2$  compatible. The limit for compatibility is thus slightly larger than $\limit$, to be precise $\free\geq481.35207212433$. At this limit, we need $\freeD= 0.69113483$ for $D_1$ and $D_2$ to be compatible. In \cite{Mattsson2008b}, it is reported that when $\free=490$, the compatibility region is $\freeD\in[0.6789094547,0.7254477238]$.

In the literature, 
three different values of $\freeD$ have been used, and the corresponding operators have optimal properties of different perspectives. 
\begin{enumerate}
    \item In \cite{ref:STRA94}, 
    $\freeD=89387/129600\approx 0.6897$ is used to obtain an operator $D_1$ with minimum bandwidth. More precisely, there are five nonzeros in the first row, instead of six with other choices of $\freeD$. The corresponding operator $D_1$ is compatible with $D_2$ if $\free\geq 481.3588804669321$.
    \item In \cite{Mattsson2004503},  
    $\freeD=342523/518400\approx 0.6607$ is used to optimize accuracy. On the first six grid points the truncation error of $D_1$ is designed to be third order. With this particular choice of $\freeD$, on the sixth grid point the truncation error is fourth order. The truncation errors on the first five grid points remain third order. However, the corresponding operator $D_1$ is not compatible with $D_2$ for any $\free$.    
    \item In \cite{Mattsson2008b},  
    $\freeD=331/472\approx 0.7013$ is derived to optimize $L^2$ accuracy and spectrum. 
    The corresponding operator $D_1$ is compatible with $D_2$ if $\free\geq 481.6401641339156$.
\end{enumerate}

\begin{remark}
With the above three choices of $\freeD$, the operator $Q+\el\el^\trans$ has full rank, see Theorem 4.4 and 4.5 in  
\cite{RUGGIU_eigen}. As a consequence, the operator $A^{wide}=D_1^\trans H D_1$ has only one zero eigenvalue, where $A^{wide}$ is part of the wide-stencil second derivative SBP operator
$
    D_2^{wide}=D_1^2=H^{-1}QD_1=H^{-1}(-D_1^\trans H+\er\er^\trans-\el\el^\trans)D_1
$.
\end{remark}

\section{The influence of $\free$ on numerical properties}
\label{InfluenceOfFree}

We have seen that the sixth order accurate narrow-stencil operator has a free parameter, $\free$. 
Moreover, from what we have seen in Figure~\ref{LargestSmallestEigOfA20} the original value $\free=490$ derived in \cite{Mattsson2004503} might not be the optimum choice, at least not when it comes to the spectral radius of $A$. However, it is not only the properties of $A$ that are important, but also those of $D_2$ and the resulting numerical solutions. Below we  investigate numerical properties, such as accuracy and stiffness, for an improved operator $D_2$.

\subsection{Truncation error of $D_2$}

As before, let $\restrict=[\smooth(x_0)\ \smooth(x_1)\ \hdots\ \smooth(x_{n})]^\trans$ be the restriction of a smooth function $\smooth(x)$ to the grid. Multiplying $D_2$ by $\restrict$, we expect the resulting vector $D_2\restrict$ to approximate $\smooth_{xx}=-\force$ well, that is, we expect the residual vector
\begin{align*}
\begin{split}
\trunc=-D_2\restrict-\fh
\end{split}
\end{align*}
to be small. The sixth order accurate narrow-stencil matrix $D_2$ is designed to have sixth order of accuracy in the interior and third order of accuracy at the boundary. These requirements are fulfilled regardless of the choice of $\free$. This means that $\trunc=0$ for $\smooth=x^p$, for $p=0,1,2,3$ and $4$, by construction.

The exact expressions of the elements in $\trunc$ are obtained using Taylor expansions and involve higher derivatives of $\smooth(x_j)$ with respect to $x$. For the sixth order accurate operator, the dominating error term is proportional to $h^3\smooth^{(5)}$ and can be estimated by letting $\smooth=x^5$ such that $\smooth^{(5)}$ is a constant and the following Taylor expansion terms zero. In this case we have $\trunc=-D_2\xfat^5 + 20\xfat^3$.
Explicitly, the $(n+1)\times1$-vector $\trunc$ has value zero in the interior, at the boundaries it is 
\begin{align*}
\trunc=h^3\left[\begin{array}{c}
 (12536750-28800\free )/13649\\
- (13794742 - 28800\free)/12013\\
 ( 13929546-28800\free )/2711\\
- (13980554 - 28800\free)/5359\\
 (13950358-28800\free )/7877\\
- (13841550 - 28800\free)/43801\\
 0\\
 \vdots\\
 0\\
( 13841550-28800\free )/43801\\
- (13950358 - 28800\free)/7877\\
 ( 13980554-28800\free )/5359\\
- (13929546 - 28800\free)/2711\\
 (13794742-28800\free )/12013\\
- (12536750 - 28800\free)/13649\\
\end{array}
\right].
\end{align*}
We want to choose the free parameter $\free$ such that the boundary error vector $\trunc$ is minimized. The "size" of $\trunc$ can be measured in many different norms, for example the usual discrete $L^2$-norm $\sqrt{\trunc^\trans\trunc}$, which is minimized for
$\free\approx482.5622776076688$. Alternatively, the error vector in the SBP norm $\|\trunc\|_H=\sqrt{\trunc^\trans H\trunc}$ is minimized when $\free\approx 483.3965798037094$. Other equivalent norms may also be used, and we can expect that $\free$ is in the range 
$\frac{12536750}{28800}\leq\free\leq\frac{13980554}{28800}$,
because this is the range wherein the components of $\trunc$ pass zero. Since the lower limit is smaller than $\free^*$, we should choose 
\begin{align}\label{AccuracyD2}
\limit\leq\free\leq\frac{13980554}{28800}\approx485.4359027777778.
\end{align}






\newcommand{\pointu}{\mathbf{u}} 
\newcommand{\fel}{\boldsymbol{\epsilon}}

We will later present numerical simulations of the wave equation and the heat equation.
However, before doing so, we will look into the influence of $\free$ on the operators $\DDN$, when it comes to accuracy of the steady problem as well on stiffness (spectral radius of $\DDN$).

\subsection{Poisson’s equation with Neumann boundary conditions}
\label{SubSecNeumann}

We first consider 
Poisson's equation with Neumann boundary conditions, discretized in \eqref{Avb}. The coefficient matrix $A$ in \eqref{Avb} is singular, so we cannot solve this system uniquely.
Instead, we solve 
the system with the help of the pseudoinverse presented in \eqref{MPinvofA}.

\subsubsection{Accuracy and spectral radius of $\DN$}

To minimize the dominating errors, we again consider $\contsol=x^5$ as the manufactured exact solution and let $\pointu=\xfat^5$ be the exact solution evaluated at the grid points.
Since we cannot solve the system uniquely in this case, we only obtain the solution up to a constant. Therefore, we define the error as 
$\fel=\pointu-\frac{\ones^\trans\pointu}{n+1}\ones-\discsol$. 

In 
Figure~\ref{NeumannAccuracyAndSpectrum}(a),
we show the error, measured in different norms, as a function of $\free$.
Zooming in, we see that error $\|\fel\|_H$ is minimized when $\free\approx484.30$ 
(for the usual discrete $L^2$-norm, the error is minimized at $\free\approx484.11$, and in the maximum norm the error is minimized at $\free\approx483.30$).
\begin{figure}[ht]\centering
\subfigure[The errors 
$\|\fel\|_H$, $\|\fel\|_{\infty}$ and $\|\fel\|_{L^2}$
]{\includegraphics[width=.495\textwidth]{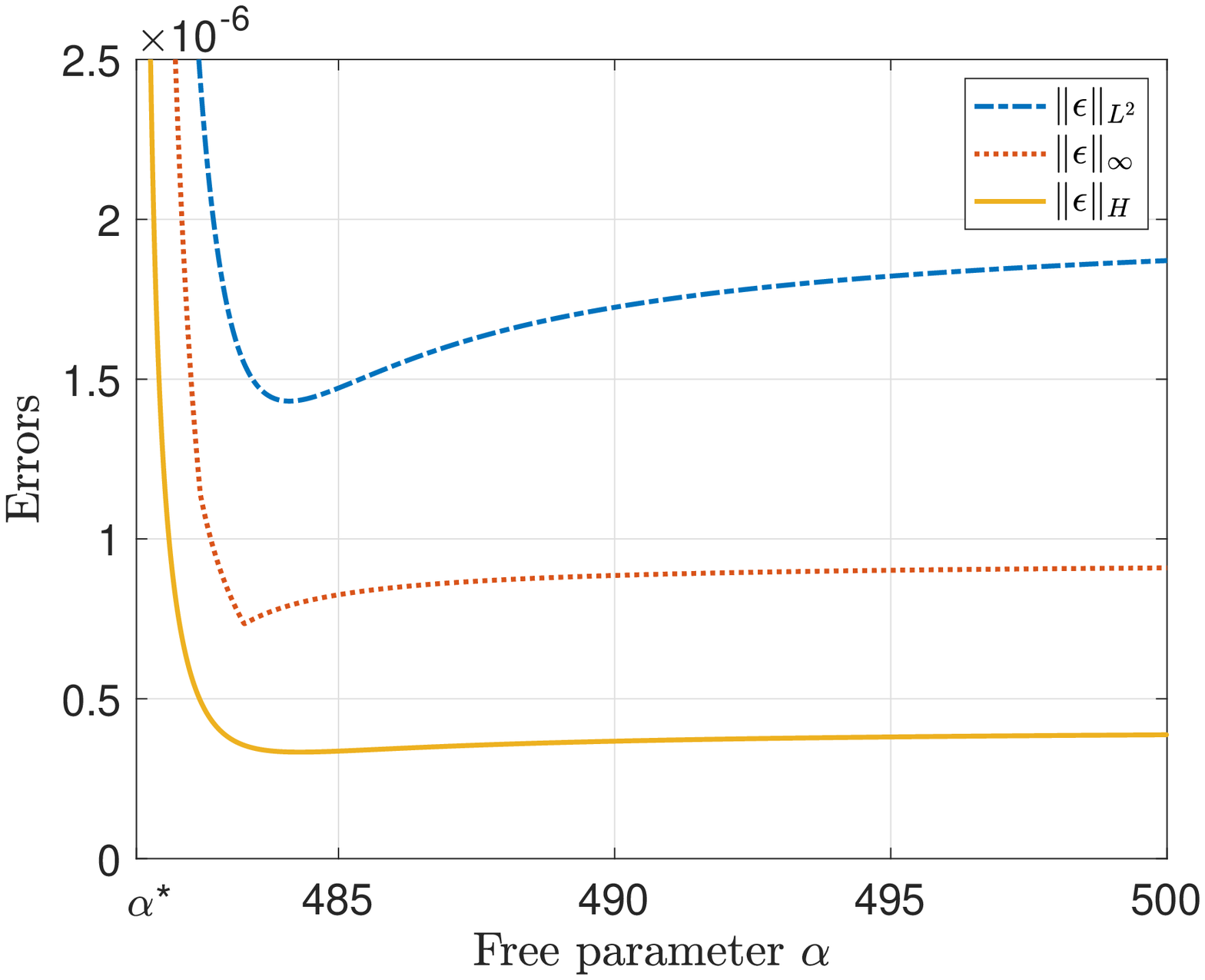}
}%
\subfigure[The spectrum 
of $-\DN$ and spectral radius $\rho(\DN)$
]{\includegraphics[width=.495\textwidth]
{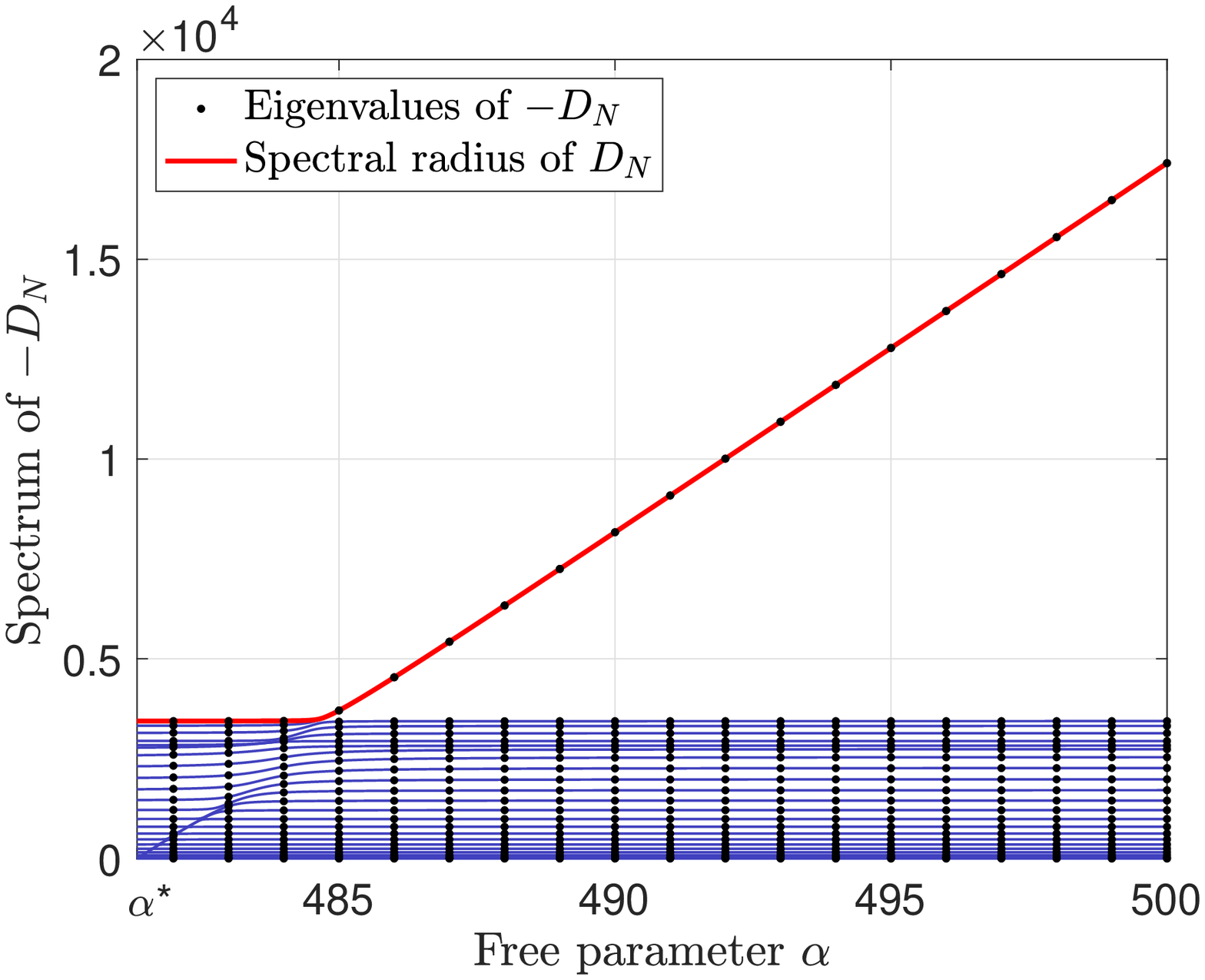}
}
\caption{Errors and numerical spectrum with Neumann boundary conditions.}
\label{NeumannAccuracyAndSpectrum}
\end{figure} 
The figure is produced using $n=24$, but the values of $\free$ that minimize the error are representative for 
all $n$. 
Note that the choice of $\free$ cannot be used to improve the order of accuracy, it only affects the error constant. 


If considering time-dependent problems, it is important that the discretization is non-stiff (at least if an explicit time-stepper is used). An indication of stiffness is the spectral radius $\rho(\DN)$, with $\DN$ given in \eqref{D2tildeNeumann}. In Figure~\ref{NeumannAccuracyAndSpectrum}(b), the eigenvalues of $\DN$ (with reversed sign) are shown as functions of $\free$. The spectral radius is the largest absolute value of these eigenvalues, marked with a red line. Looking in Figure~\ref{NeumannAccuracyAndSpectrum}(b), we see that for choices of $\free$ larger than $\free\approx484.6$, the spectral radius starts growing fast (linearly). Luckily, this is not in conflict with the accuracy demand from above and, if we are interested in minimizing the error in the SBP-norm, $\free\approx484.3$ remains a good option for the Neumann problem.

The choice of $\free=484.3$ gives a 10\% smaller error constant than the standard choice $\free=490$, and reduces the spectral radius to 42-43\% of the original value.

\subsection{Poisson's equation with Dirichlet boundary conditions}
\label{SubSecDirichlet}

Now consider Poisson's equation with Dirichlet boundary conditions, that is \eqref{wave} with \eqref{waveD} with the term $\contsol_{tt}$ removed. This problem is discretized as 
$0=\DD\discsol+\fD$, 
where $\DD$ is given in \eqref{D2tildeDirichlet}.
The properties of $\DD$ do not only depend on $\free$, but also on the choice of 
penalty strength $\factor$ in \eqref{factor}.
We will now look into how $\free$ and $\factor$ influence the accuracy 
as well as the spectral radius.

\subsubsection{Accuracy}
\label{sec_AccuracyPoissonDirichlet}

We start by investigating the accuracy. Solving the above linear system of equation and comparing with the exact solution $\pointu$, we define the solution error as $\fel=\pointu-\discsol$.
As before, we want to focus on the errors generated at the boundary since they are the ones most affected by $\free$, and let $\pointu=\xfat^5$ again.

The value of $\free$ that minimizes the error depends on the choice of penalty strength, see Figure~\ref{D2tildeDirichletAcc}(a), where the errors produced when $n=24$ are shown.
\begin{figure}[ht]\centering
\subfigure[The error 
as a function of $\free$, for various choices of $\factor$ ]{\includegraphics[width=.49\textwidth]
{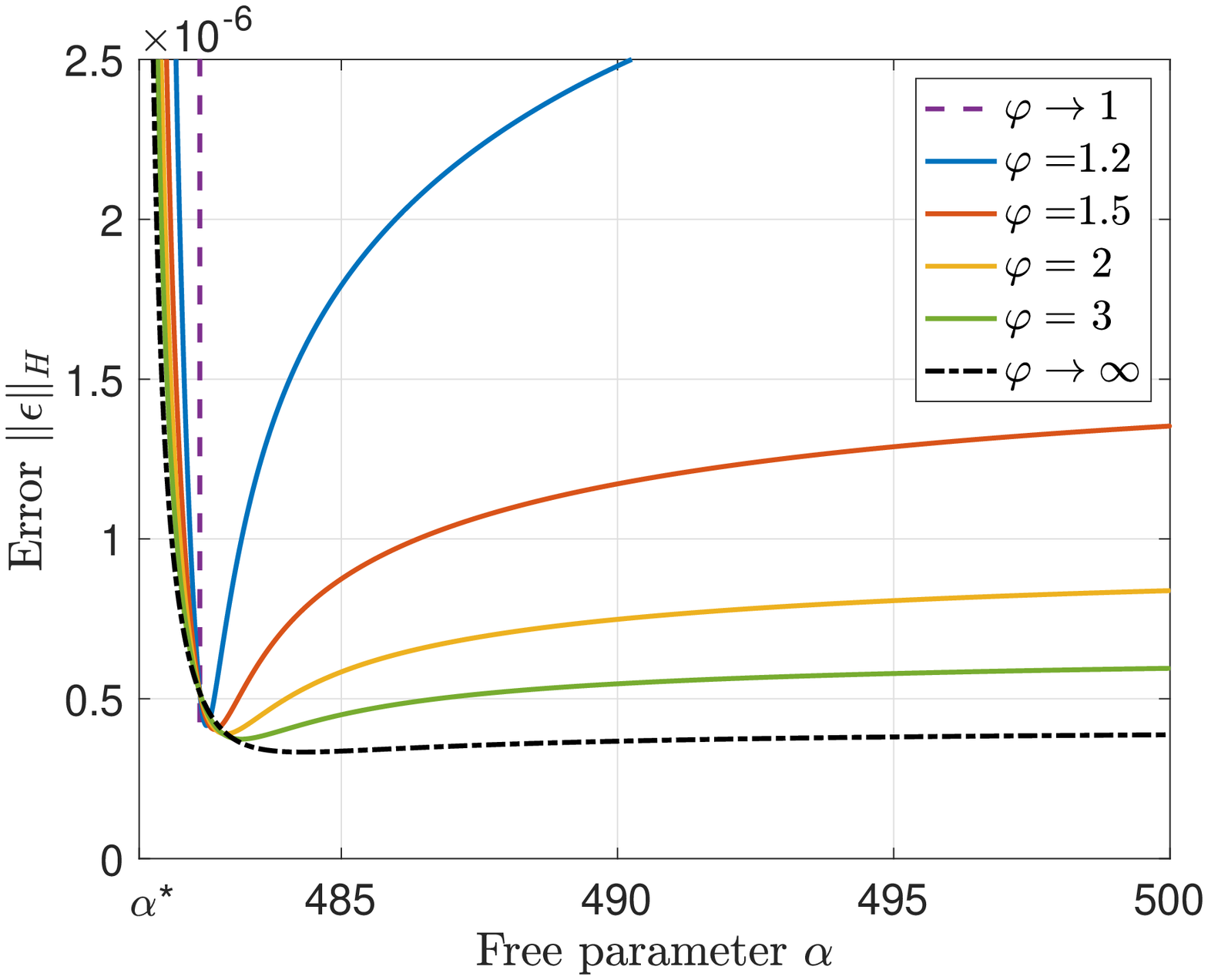}
}%
\subfigure[The error (scaled by $10^6$) as a function of $\free$ and $\factor$]{\includegraphics[width=.49\textwidth]
{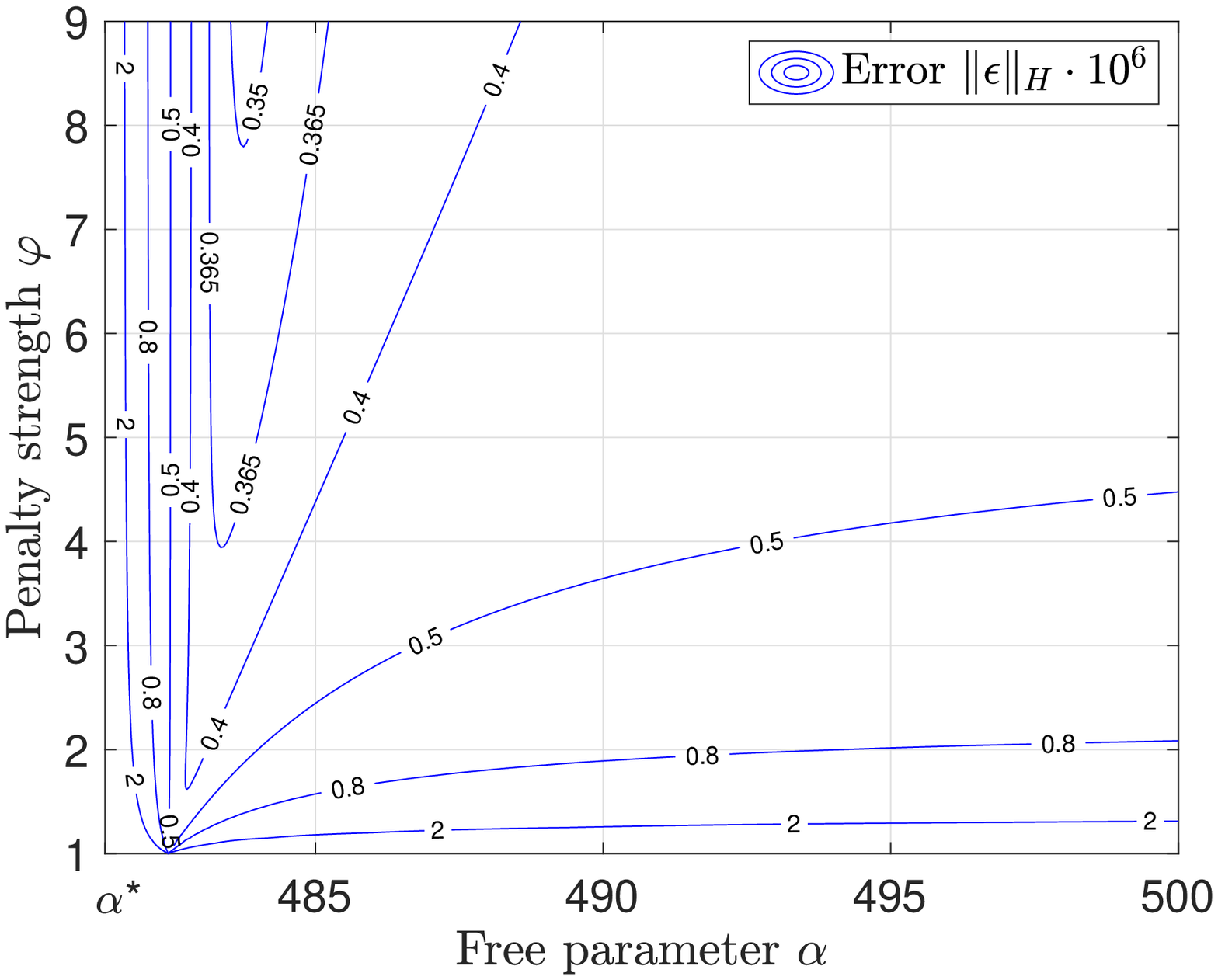}
}\caption{The error $\|\fel\|_H$ when imposing Dirichlet boundary conditions}
\label{D2tildeDirichletAcc}
\end{figure} 
If the penalty is chosen right at the stability limit (i.e. $\factor\to1$), then $\|\fel\|_H$ is minimized for $\free=482.44$. If the penalty is tuned extremely strong (i.e. $\factor\to\infty$), then $\|\fel\|_H$ is minimized for $\free=484.30$.
These values seem to be independent of the number of grid points $n$, such that the most accurate solutions of the problem with Dirichlet boundary conditions are obtained for
\begin{align}\label{FreeAccDir}
482.44\leq\free\leq484.30,
\end{align}
for any given choice of $\factor>1$.
As mentioned, the SBP-norm is probably the most natural choice and the norm we will pay most attention to in the rest of the paper -- however, if instead minimizing with respect to some other norm, the results change slightly (in the usual discrete $L^2$-norm, the error is minimizied for $482.44\leq\free\leq484.11$, and in the maximum norm for $482.44\leq\free\leq483.30$).
Interestingly, the values of $\free$ minimizing the error as $\factor\to\infty$, are the same values that 
minimize the errors when using the Neumann boundary conditions. 

By increasing the strength of the penalty, the error usually decreases, see Figure~\ref{D2tildeDirichletAcc}(b), 
where the contour lines of the error $\|\fel\|_H$ are shown as functions of $\free$ and $\factor$. 
The smallest errors are found for $\free=484.30$ with $\factor\to\infty$. 
However, as we will discuss next, if the penalty strength is too high, the scheme becomes stiff.

\subsubsection{Spectral radius of $\DD$}

In Figure~\ref{fig:SpectrumDirichlet}, we see the spectrum of $\DD$. Looking in the figure it is not distinguishable, but the {\it two} largest eigenvalues (one for each boundary) are following the red line (the spectral radius). Moreover, we see that not only they increase with increasing $\free$, also the third and fourth largest eigenvalue do (the two top blue lines -- also indistinguishable from each other -- separating from the rest). This is particularly apparent for larger values of $\factor$, such as in Figure~\ref{fig:SpectrumDirichlet}(b). In Figure~\ref{fig:SpectrumDirichlet}(a), the lines following the eigenvalues have a less "wavy" behaviour, because two of the eigenvalues are fixed at zero (since $\factor=1$ makes $\DD$ singular).
\begin{figure}[ht]
 \centering
\subfigure[$\factor=1$]{
 \includegraphics[width=.49\textwidth]{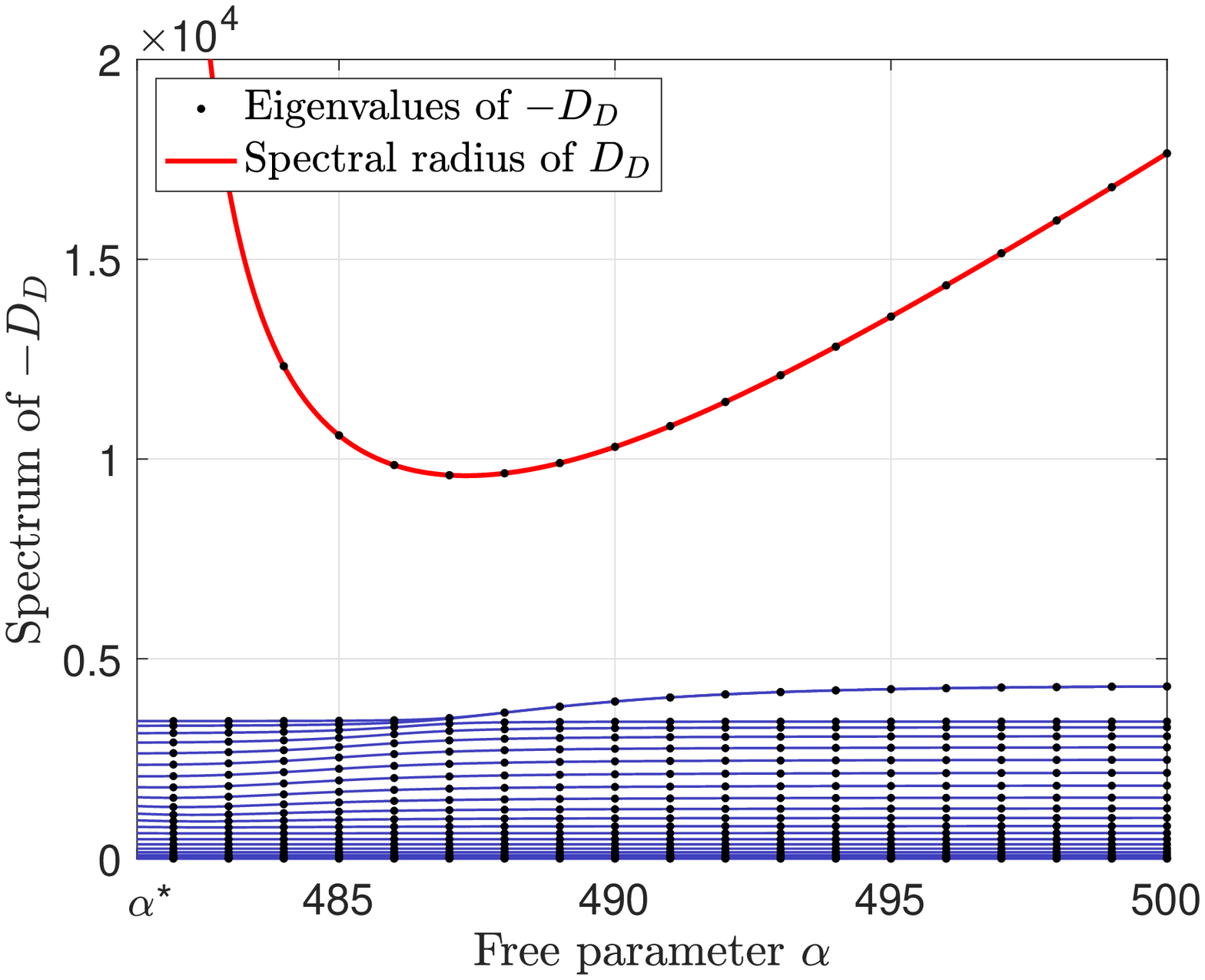}}%
 \subfigure[$\factor=2$]{
 \includegraphics[width=.49\textwidth]{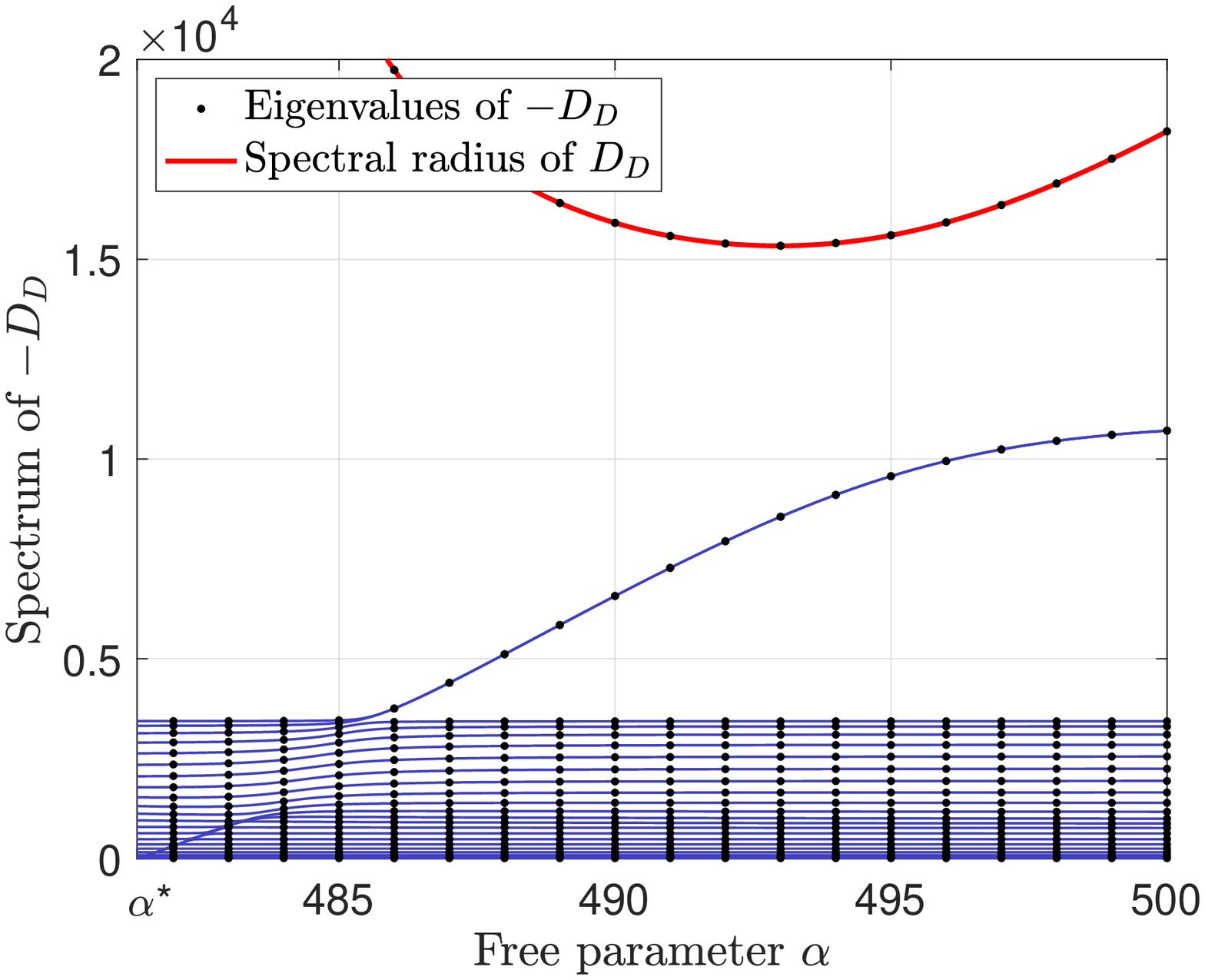}}
 \caption{Eigenvalues of the discretization operator with Dirichlet boundary conditions
 }
 \label{fig:SpectrumDirichlet}
\end{figure}

As mentioned, 
the spectral radius of an operator is an indication of stiffness. The spectral radius $\rho(\DD)$ is shown in Figure~\ref{Fig7_SpectralRadiusD2tildeDirichlet}(a), as a function of $\free$, for different choices of $\factor$.
When $\factor=1$, the spectral radius is minimized for $\free\approx487.30$, and as $\factor$ increases so do the spectral radius as well as the value of $\free$ that minimizes the spectral radius.
We note that for any given choice of $\factor\geq1$, $\free\geq487.30$ is necessary to minimize the spectral radius. 
\begin{figure}[ht]\centering
\subfigure[$\rho(\DD)$ as a function of $\free$ for various $\factor$ (for $n=24$)]{\includegraphics[width=.495\textwidth]
{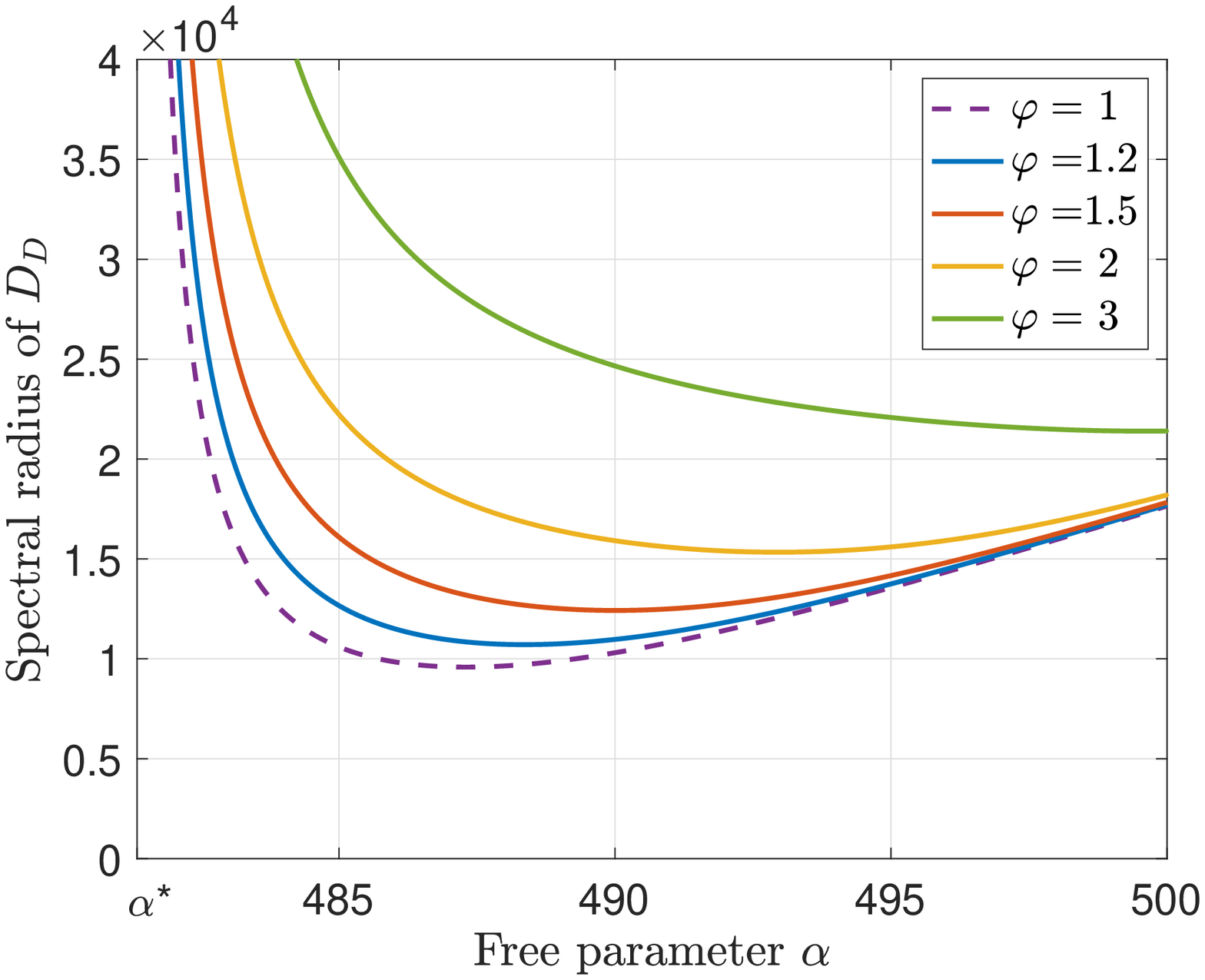}
}\subfigure[$\rho(\DD)$ as a function of $\free$ and $\factor$ (for $n=24$)]{\includegraphics[width=.495\textwidth]
{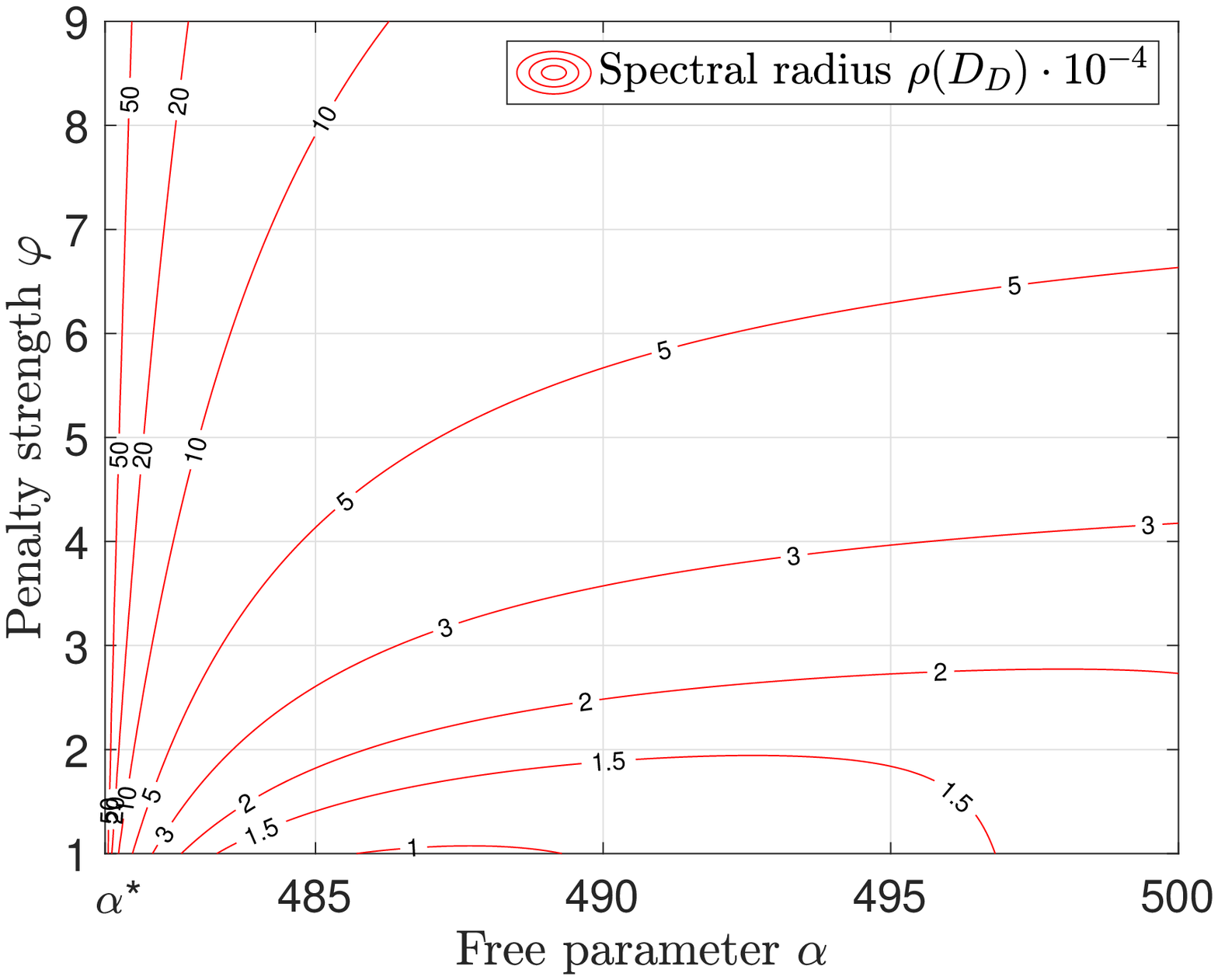}}
\caption{The spectral radius $\rho(\DD)$ when imposing Dirichlet boundary conditions.}
\label{Fig7_SpectralRadiusD2tildeDirichlet}
\end{figure}
In Figure~\ref{Fig7_SpectralRadiusD2tildeDirichlet}(b), we see more clearly that $\DD$ has the smallest spectral radius $\rho(\DD)$ for $\free\approx487.30$ with $\factor=1$, and increases with increasing penalty strength.  A strong penalty with $\free\to\limit$ is the worst combination.

\subsubsection{A weighted balance between accuracy and spectral radius of $\DD$}

So far, we have seen that $\DD$ gives best accuracy for $482.44\leq\free\leq484.30$, preferably $\free\approx484.30$ with $\factor\to\infty$, see Figure~\ref{D2tildeDirichletAcc}(b). However, $\DD$ has the smallest spectral radius $\rho(\DD)$ for $\free\approx487.30$ with $\factor=1$, see Figure~\ref{Fig7_SpectralRadiusD2tildeDirichlet}(b). These specifications are unfortunately mutually exclusive, and we would like to find a compromise. 

How $\free$ and $\factor$ should be combined to obtain an "optimal" operator $\DD$, depends on how low errors versus low stiffness are valued. Here, we will contrast the relative error $\|\fel\|_H/\min_{\free,\factor}(\|\fel\|_H)$ and the relative spectral radius $\rho/\min_{\free,\factor}(\rho)$.
Consider Figure~\ref{SlidingOptimum}(a), where the contour lines of the relative error are shown in blue and the contour lines of the relative spectral radius are shown in red (note that compared to Figure~\ref{D2tildeDirichletAcc}(b) and Figure~\ref{Fig7_SpectralRadiusD2tildeDirichlet}(b), the values have been re-scaled).
\begin{figure}[ht]\centering
\subfigure[Optimal combinations of $\free$ and $\factor$]{\includegraphics[width=.5\textwidth]{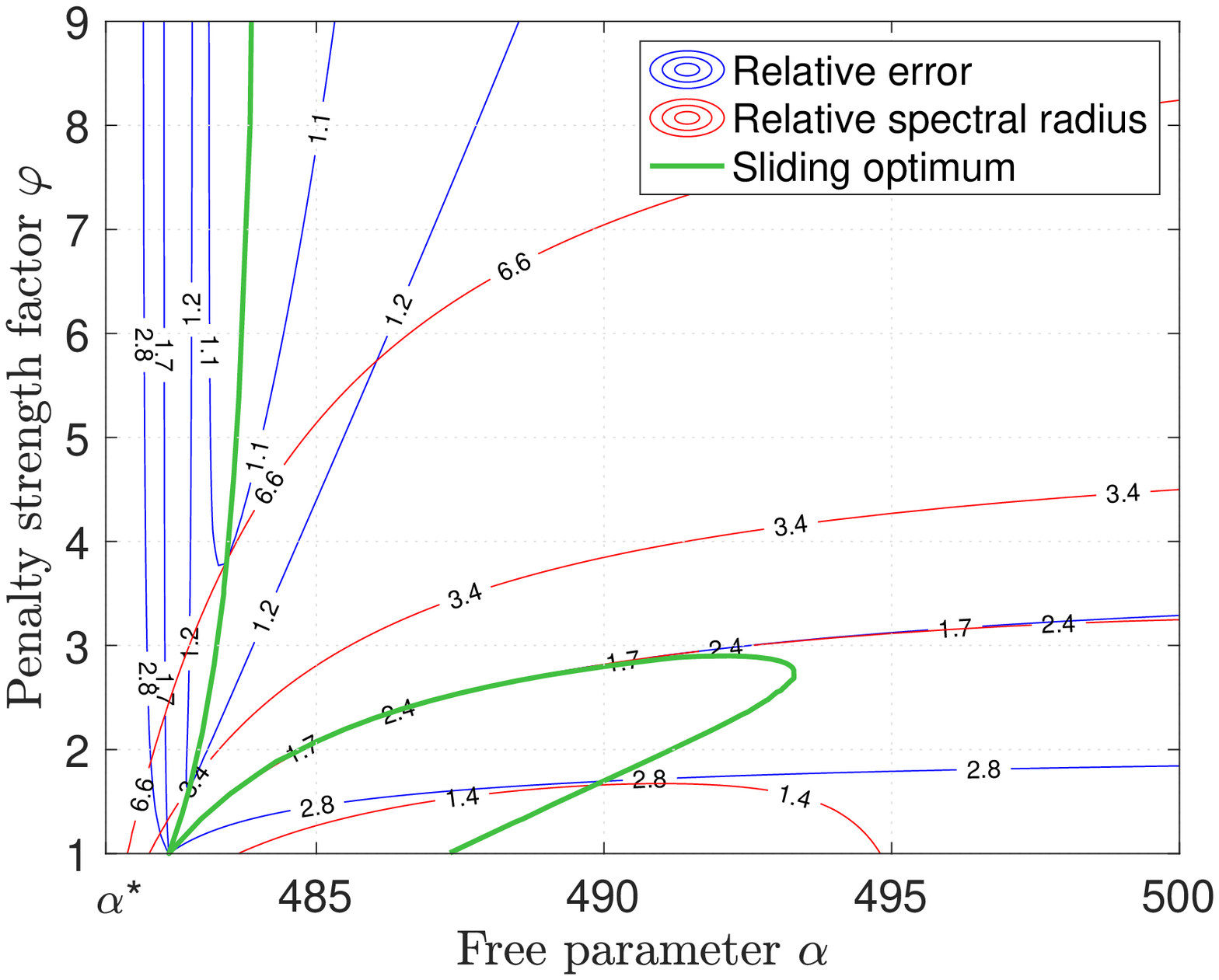}}%
\subfigure[Resulting error and spectral radius for given pairs $(\free,\factor)$]{\includegraphics[width=.5\textwidth]{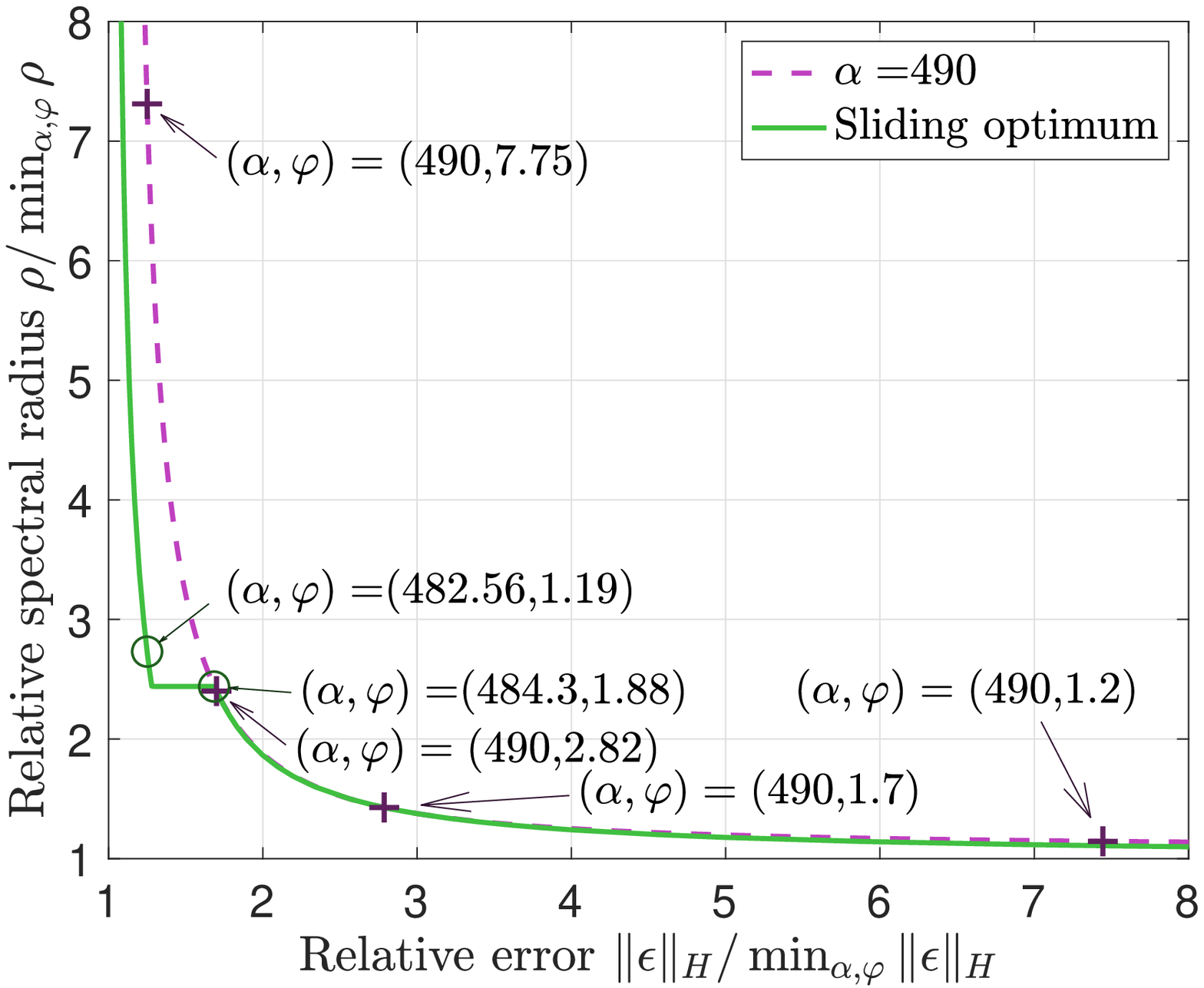}}
\caption{The relative error $\|\fel\|_H/\min_{\free,\factor}\|\fel\|_H$ and the relative spectral radius $\rho/\min_{\free,\factor}\rho$ when imposing Dirichlet boundary conditions.
In (a): The relative error and spectral radius are shown as blue and red contour lines, respectively. The green line shows where the minimum of $\|\fel\|_H\rho(\DD)^\omega$ is found, where the weight $\omega$ varies.
In (b): The relative error and relative spectral radius obtained given particular combinations $(\free,\factor)$, comparing the optimal combinations from (a) with the fixed $\free=490$ case. 
}
\label{SlidingOptimum}
\end{figure} 
Now suppose that accuracy is most important to us, but we are willing to accept an error that is 10\% higher than the minimum. 
Then we move along the blue contour line marked with 1.1 until we find a minimum of the relative spectral radius, which in this case is approximately 6.61. Thus the optimal combination of $\free$ and $\factor$ is found in the intersection between the 1.1 blue contour line and a 6.61 red contour line, at $(\free,\factor)\approx(483.44,3.82)$. If we instead are prepared to accept a 20\% increase in the error constant to get a lower spectral radius, the lowest possible (relative) spectral radius is 3.39, obtained for $(\free,\factor)\approx(482.80,1.64)$. In this way, we find a sliding optimum, marked as a green line in Figure~\ref{SlidingOptimum}(a) and with some examples given in the first four columns in
Table~\ref{tab:optimumcombinationsALT}.
\begin{table}[H]
 \centering
 \begin{tabular}{|l|rlr|rlr|}
 \hline
 $\|\fel\|_H/\min_{\free,\factor}\|\fel\|_H$ & $\rho/\min_{\free,\factor}\rho$&$\free$&$\factor$&$\rho/\min_{\free,\factor}\rho$& $\free$&$\factor$\\\hline
 $\to1$&$\to\infty$&484.29&$\to\infty$&\multicolumn{3}{c|}{not possible}\\
1.05&12.97 & 483.85 & 8.01&\multicolumn{3}{c|}{not possible}\\
1.1 & 6.61 & 483.44 & 3.82&\multicolumn{3}{c|}{not possible}\\
1.15&4.47 & 483.10 & 2.38&21.03&490&21.32\\
1.2 & 3.39& 482.80 & 1.64&10.66&490&11.06\\
1.25 & 2.73 & 482.56 & 1.19&7.31&490& 7.75 \\
1.7 & 2.40 & 490.01 & 2.82& 2.40 & 490 & 2.82\\
2 & 1.86 & 492.48 & 2.36&1.88&490&2.25\\
2.79& 1.43& 490.00& 1.70& 1.43& 490& 1.70\\
7.5 & 1.11 & 487.99 & 1.18&1.14&490&1.20\\
$\to\infty$&$\to1$&487.30&$\to1$&1.08&490&$\to1$\\\hline
 \end{tabular}
 \caption{Examples of "optimal" combinations of $\free$ and $\factor$, given the prescribed relative error in the left column.
 As a comparison, in the right columns we show how the penalty strength must be chosen to achieve the same accuracy using the standard $\DD$ with $\free=490$. When $\free=490$, the smallest relative error possible is 1.11. 
 Here $n=24$ has been used.
}
 \label{tab:optimumcombinationsALT}
\end{table}


Another way of viewing it, is that we look for where the minimum of $\|\fel\|_H\rho(\DD)^\omega$ is found, where the weight $0<\omega<\infty$ varies.
Zooming in on the green line in Figure~\ref{SlidingOptimum}(a), we see that the optimal (with respect to accuracy and stiffness with varying weight) is found for
\begin{align}\label{WeightedOptimal}
 482.44\leq\free\leq493.31,
\end{align}
with a matching choice of penalty strength $\factor$.
Repeating this procedure for the usual discrete $L^2$-norm, the optimal is found for $482.44\leq\free\leq493.52$, and in the maximum norm for $482.44\leq\free\leq492.86$.

From Table~\ref{tab:optimumcombinationsALT}, it is clear that if the errors decrease, the spectral radius increase (and vice versa). This is also illustrated in Figure~\ref{SlidingOptimum}(b), where the dashed curve shows the resulting error and spectral radius when $\free=490$ and $\factor$ is varied. The same is then done following the path of the "sliding optimum" from Figure~\ref{SlidingOptimum}(a), yielding a 
curve that always has {\it at least} as small error and spectral radius as when $\free$ is fixed. 

In Figure~\ref{SlidingOptimum}(b), we note that in the intermediate region, the standard value $\free=490$ is actually an excellent choice, not leaving much room for improvement. 
Most difference is observed in the region with small errors to the left in the figure. 
For example, the choices $(\free,\factor)=(490,7.75)$ and $(\free,\factor)=(482.56,1.19)$ both give the relative error $1.25$, but the former gives an relative spectral radius $7.31$ and the latter only $2.73$.
This is in agreement with what can be observed in Table~\ref{tab:optimumcombinationsALT}.

In Figure~\ref{SlidingOptimum}, there appears to be some sort of singularity at the point $(\free,\factor)\approx(482.44,1)$. Here the relative stiffness is around 2.44, but the relative error varies rapidly between 1.28 and 1.68, even when $\free$ is just altered slightly.
The cause for this is that the error  is more or less independent of $\factor$ for $\free\approx482.44$ (this can be seen in Figure~\ref{D2tildeDirichletAcc}(b), noting that the blue line $0.5\cdot10^{-6}$ is almost vertical at $\free\approx482.44$). On the other hand, $\factor=1$ makes $\DD$ singular.
The area around $(\free,\factor)\approx(482.44,1)$ is thus quite intriguing: This is where most "accuracy per stiffness" can be gained, but since it is sensitive to parameter changes one should probably leave some margin 
-- for example by picking 
$(\free,\factor)=(482.56,1.19)$, as in Figure~\ref{SlidingOptimum}(b).

Note that the optimal values discussed above and shown in \eqref{WeightedOptimal}, in Figure~\ref{SlidingOptimum} and Table~\ref{tab:optimumcombinationsALT} are obtained by balancing the error measured in the $H$ norm and the spectral radius of the discretization matrix. If the error is measured in another norm the numbers and curves change accordingly.

\subsection{Mixed boundary conditions}

As mentioned above, in the intermediate region of Figure~\ref{SlidingOptimum}(b), we did not get much improvement for the Dirichlet boundary conditions by changing $\free$. Another way of viewing it, is that it does not get any worse either. This can be used to choose $\free$ to make $D_2$ more of a "multipurpose operator".

In Figure~\ref{SlidingOptimum}(a), we note that the contour lines of the relative error and the relative spectral radius are almost aligned when the relative error is around 1.7 and the relative spectral radius is around 2.4. 
For example, the combination $(\free,\factor)=(490,2.82)$ yields the relative error 1.70 and the relative spectral radius 2.40, while the combination $(\free,\factor)=(484.3,1.88)$
 yields the relative error 1.68 and the relative spectral radius 2.44.
This means that the choice $(\free,\factor)=(484.3,1.88)$ is equivalent to the choice $(\free,\factor)=(490,2.82)$, in the sense that both of these combinations give an operator $\DD$ with almost the same relative error constant 
and relative spectral radius.
In other words, the standard choice of $\DD$ with a rather strong penalty can be replaced by the  choice that is optimal for the Neumann operator $\DN$ -- without changing the error constant or the spectral radius of $\DD$. 
This could be useful if we want optimal accuracy when having Neumann boundary conditions and accept a relative error and spectral radius around 2 when having Dirichlet boundary conditions.

Finally, we also consider Poisson’s equation with one Dirichlet boundary condition at $x = 0$ and one Neumann boundary condition at $x = 1$. Repeating the process (presented above for the pure Dirichlet case) shows that the influence of the Dirichlet boundary condition is dominating -- the spectral radius is almost identical when  comparing  Figure~\ref{fig:OneOfEach}(a) with Figure~\ref{fig:SpectrumDirichlet}(b).
 \begin{figure}[hbt]
 \centering
 \subfigure[Spectrum for $\factor=2$ and $n=24$. ]{
 \includegraphics[width=.49\textwidth]{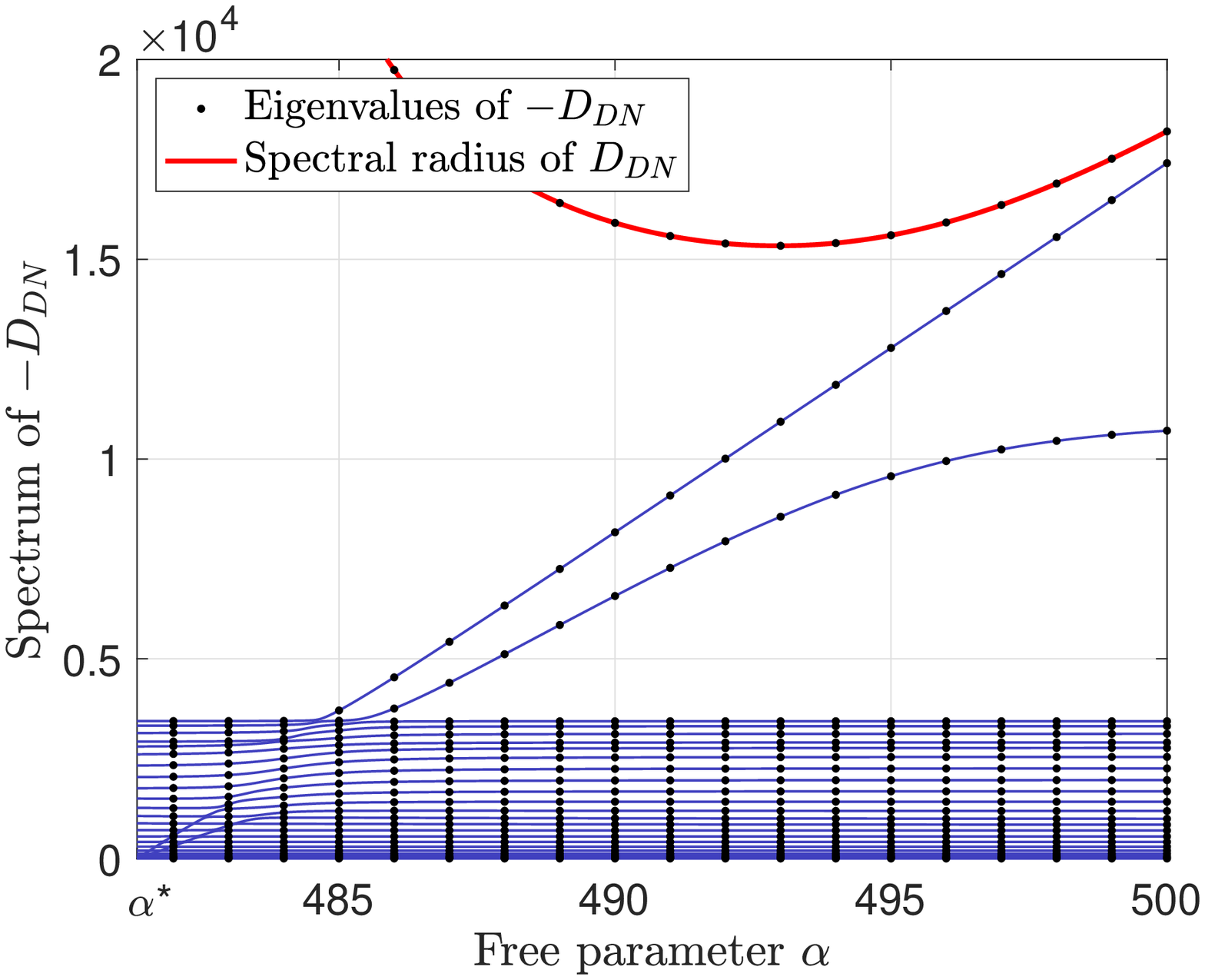}}
 \subfigure[Optimal combinations of $\free$ and $\factor$. 
 ]{
 \includegraphics[width=.49\textwidth]{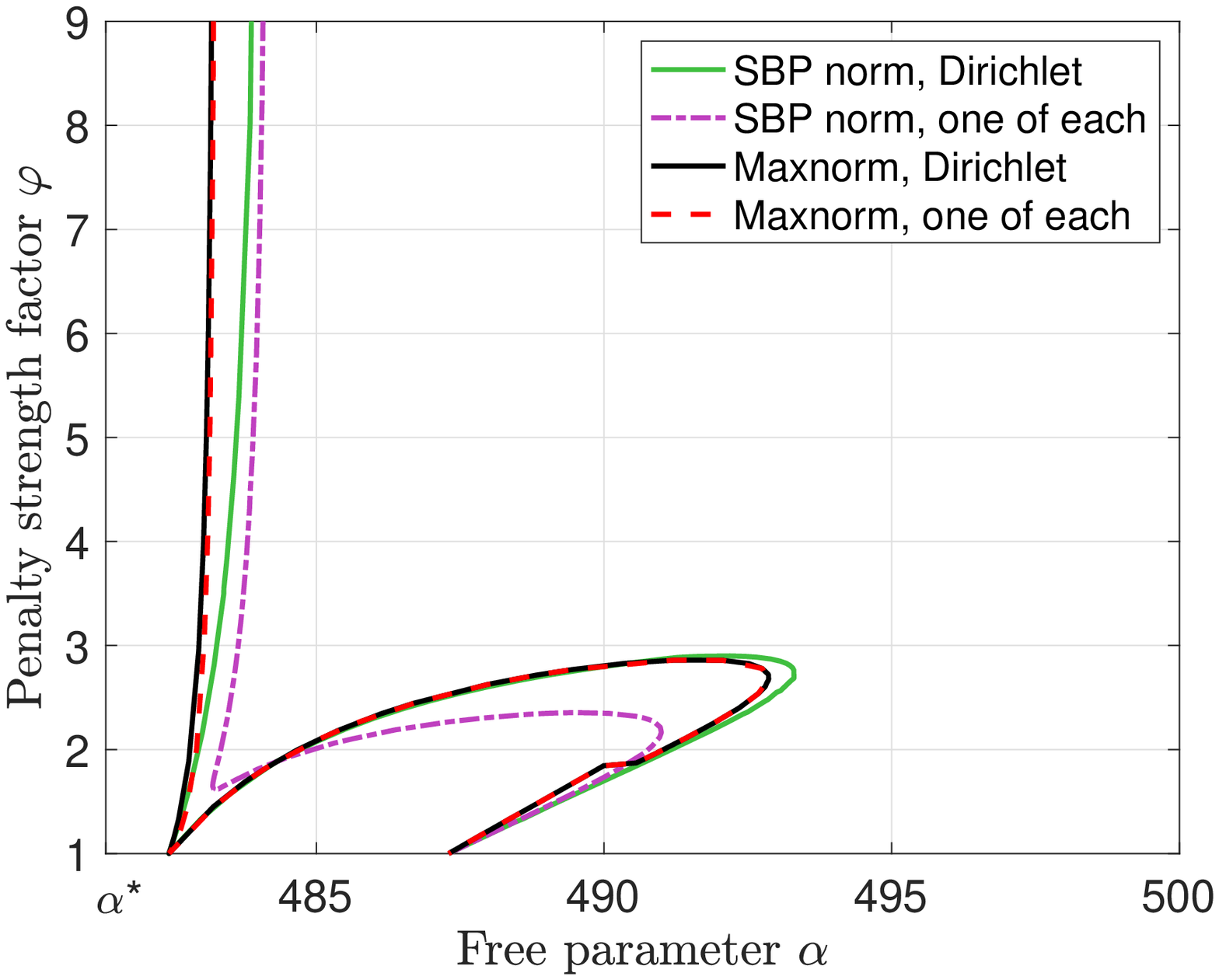}
 }
 \caption{Numerical spectrum and optimal combinations of $(\free,\factor)$ obtained when imposing one Dirichlet boundary condition and one Neumann boundary condition. }
 \label{fig:OneOfEach}
\end{figure}
The errors also have similar overall behaviour as in the pure Dirichlet case. 
In the end, we obtain a similar "sliding optimum" as in Figure~\ref{SlidingOptimum}(a) with the values  shifted compared to the pure Dirichlet case, see the purple dash-dotted curve in Figure~\ref{fig:OneOfEach}(b).
The two curves intersect at $(\free,\factor)=(484.3,1.88)$ and join when $(\free,\factor)\to(484.3,\infty)$ and $(\free,\factor)\to(487.3,1)$.

For the Dirichlet boundary conditions we have for clarity focused on the errors in the SBP norm, but the behavior is similar in other norms. Producing a sliding optimum for the maximum norm gives a quite similar result as the green curve, both when having Dirichlet boundary conditions on both boundaries or on one only, see the black solid and the red dashed curves in Figure~\ref{fig:OneOfEach}(b). 
All in all, having one boundary condition of each type gives us the same kind of balance problem as in the pure Dirichlet case -- what combinations of $\free$ and $\factor$ is considered the best depends on how low errors versus low stiffness are valued.

\subsection{Choosing an optimal value of $\free$ 
}\label{sec_rec}

For convenience, before doing time dependent simulations, we gather all our information about the  free parameter $\free$ in one place, see Table~\ref{tab:DemandsOnFree}.
\begin{table}[ht]
 \centering
 \begin{tabular}{|l|l|l|}
 \hline
 Requirement & Reason & Reference\\\hline
 $\free\geq481.3408873321106$ & Stability: To make $A\geq0$&Equation \eqref{freelimit}\\
 $\free\geq481.35207212433$&Compatibility: Possibility of making $D_1$ and $D_2$ compatible&Figure~\ref{fig:CompatibilityRegion2}\\
 $\free\geq 481.6401641339156$&Compatibility between $D_2$ and $D_1$ with $\freeD=331/472$& Section~\ref{SecComp}\\\hline
 $\free\leq485.4359027777778$&Accuracy: To make $D_2$ optimally accurate (in any norm) &Equation \eqref{AccuracyD2}\\\hline
 $\free\approx483.30$&Accuracy, Neumann: Minimizes the maximum norm of $\fel$&Section~\ref{SubSecNeumann}\\
 $\free\approx484.11$&Accuracy, Neumann: Minimizes the $L^2$-norm of $\fel$&Section~\ref{SubSecNeumann}\\
 $\free\approx484.30$
 &Accuracy, Neumann: Minimizes $\|\fel\|_H$&Figure~\ref{NeumannAccuracyAndSpectrum}(a)\\
 $\free\leq484.6$&Stiffness, Neumann: Keeps the spectral radius of $\DN$ small&Figure~\ref{NeumannAccuracyAndSpectrum}(b)\\
 \hline
 $482.44\leq\free\leq483.30$&Accuracy, Dirichlet: Minimizes the maximum norm of $\fel$&Section~\ref{SubSecDirichlet}\\
 $482.44\leq\free\leq484.11$&Accuracy, Dirichlet: Minimizes the $L^2$-norm of $\fel$&Section~\ref{SubSecDirichlet}\\
 $482.44\leq\free\leq484.30$&Accuracy, Dirichlet: Minimizes $\|\fel\|_H$ & Equation \eqref{FreeAccDir}\\
 $\free\geq487.30$&Stiffness, Dirichlet: To minimize the spectral radius of $\DD$&Figure~\ref{Fig7_SpectralRadiusD2tildeDirichlet}\\
 $482.44\leq\free\leq493.31$&Weighted compromise between small $\|\fel\|_H$ and small $\rho(\DD)$&Equation~\eqref{WeightedOptimal}\\
 \hline
 \end{tabular}
 \caption{Summary of demands and requests on $\free$}
 \label{tab:DemandsOnFree}
\end{table}
We want to stress that we have not tested all aspects and properties of $D_2$ and can therefore not single out a "best" value of $\free$ for a general purpose $D_2$.
However, 
when optimizing for a specific purpose, 
the free parameter offers a possibility to improvement.
%
Based on the investigations we have done, taking stability, compatibility, accuracy and spectral radius into consideration for the scalar Poisson's equation, we recommend to use $482.44\leq\free\leq493.31$. 
Depending on application, we suggest:
\begin{itemize}
\item For Neumann boundary conditions: Use $\free\approx484$ to minimize both the error and spectral radius.
\item For Dirichlet boundary conditions or mixed boundary conditions: Recall that the penalty strength $\factor$ from \eqref{factor} must be taken into consideration when choosing $\free$. That is, use a combination of $(\free,\factor)$ from Figure~\ref{SlidingOptimum} (exemplified in 
Table~\ref{tab:optimumcombinationsALT}) or Figure~\ref{fig:OneOfEach}(b). 
The standard value $\free=490$ gives a good balance between error and spectral radius with 
$1.7\lesssim\factor\lesssim2.8$.
If a smaller spectral radius is important, choose  $(\free,\factor)\approx(488,1.2)$.  If higher accuracy is desired,  choose 
for example $(\free,\factor)\approx(483,2)$.
\end{itemize}
Finally, a reminder: If opting for using any $\free\neq490$, keep in mind that the borrowing capacity $\app$ needs to be adjusted accordingly. Using \eqref{ComputeApp}, we obtain for example
\begin{align*}
\app&=0.087556118235046,&&\text{for }\free=483\\
\app&=0.187871502626966,&&\text{for }\free=490
\end{align*}
for $n\geq21$.

Next, we are going to see how the operators $\DDN$ behave in time-dependent problems, exemplified first with the heat equation and thereafter the wave equation.


\subsection{The heat equation}

We consider 
the heat equation $\contsol_{t}=\contsol_{xx}+\force$ with Neumann boundary conditions, discretized as \eqref{SchemeTime} with the term $\discsol_{tt}$ replaced by $\discsol_{t}$.
As the exact (manufactured) solution, we use 
 \begin{align*}
 \contsol(x,t)=\frac{\sin( cx+2c^2t)e^{ c(x-1)} +\sin(-cx+2c^2t)e^{-c(x-1)}}{e^c+e^{-c}},
 \end{align*} 
which has zero forcing function. As initial data to the numerical solution, we use the restriction of $\contsol(x,0)$ to the grid $\xfat$. 
We have used the implicit Runge--Kutta solver LobattoIIIC for the time-stepping, which is fourth order accurate and without stability restrictions on the time-step. The time-steps are chosen small enough such that the spatial errors dominate.
\begin{figure}[hbt]
\centering
\subfigure[End time 10, time steps 40000, $n=30$
]{\includegraphics[width=0.5\textwidth]{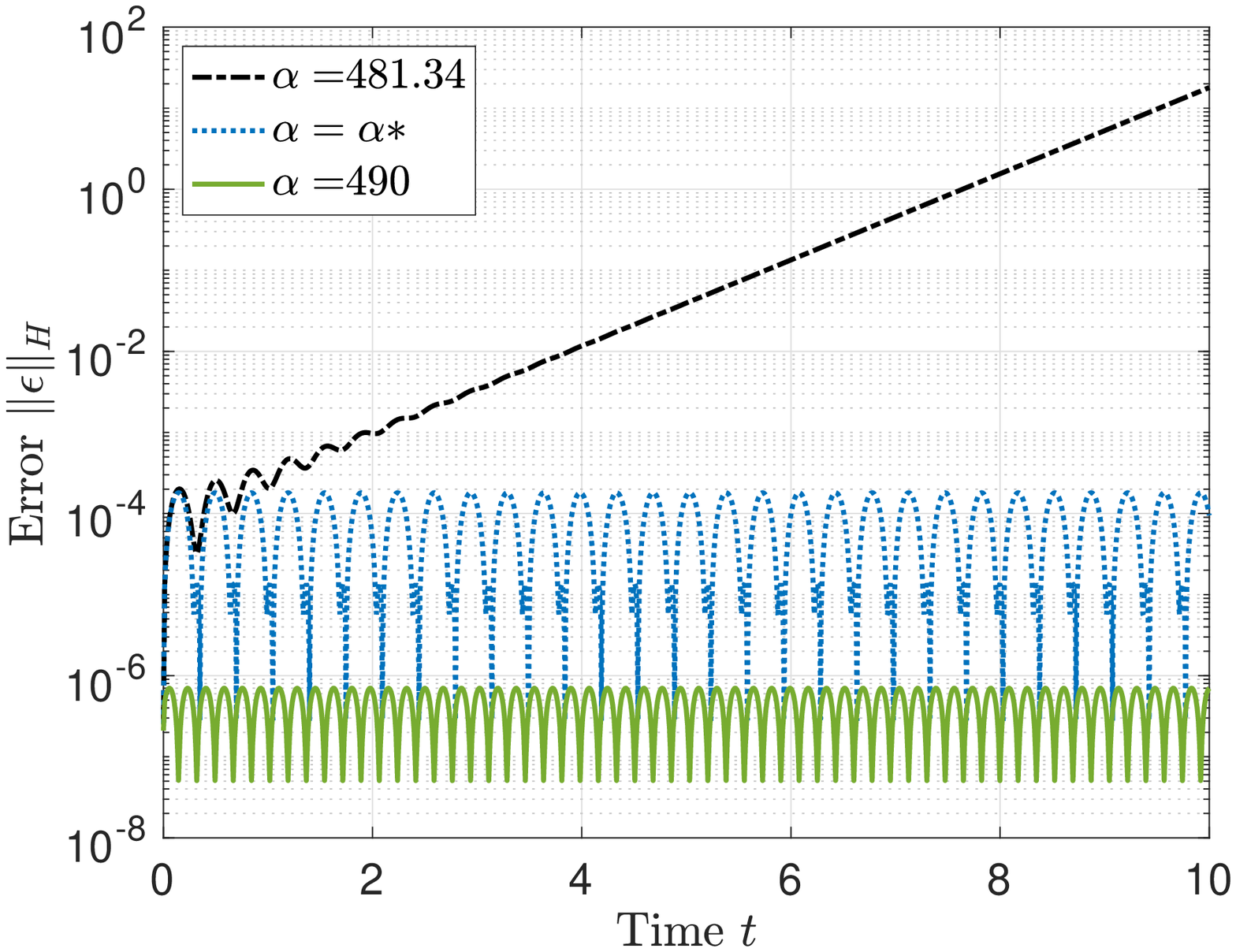}}%
\subfigure[End time 1, time steps 4000, $n=25, 50, 100, 200$]{\includegraphics[width=0.5\textwidth]{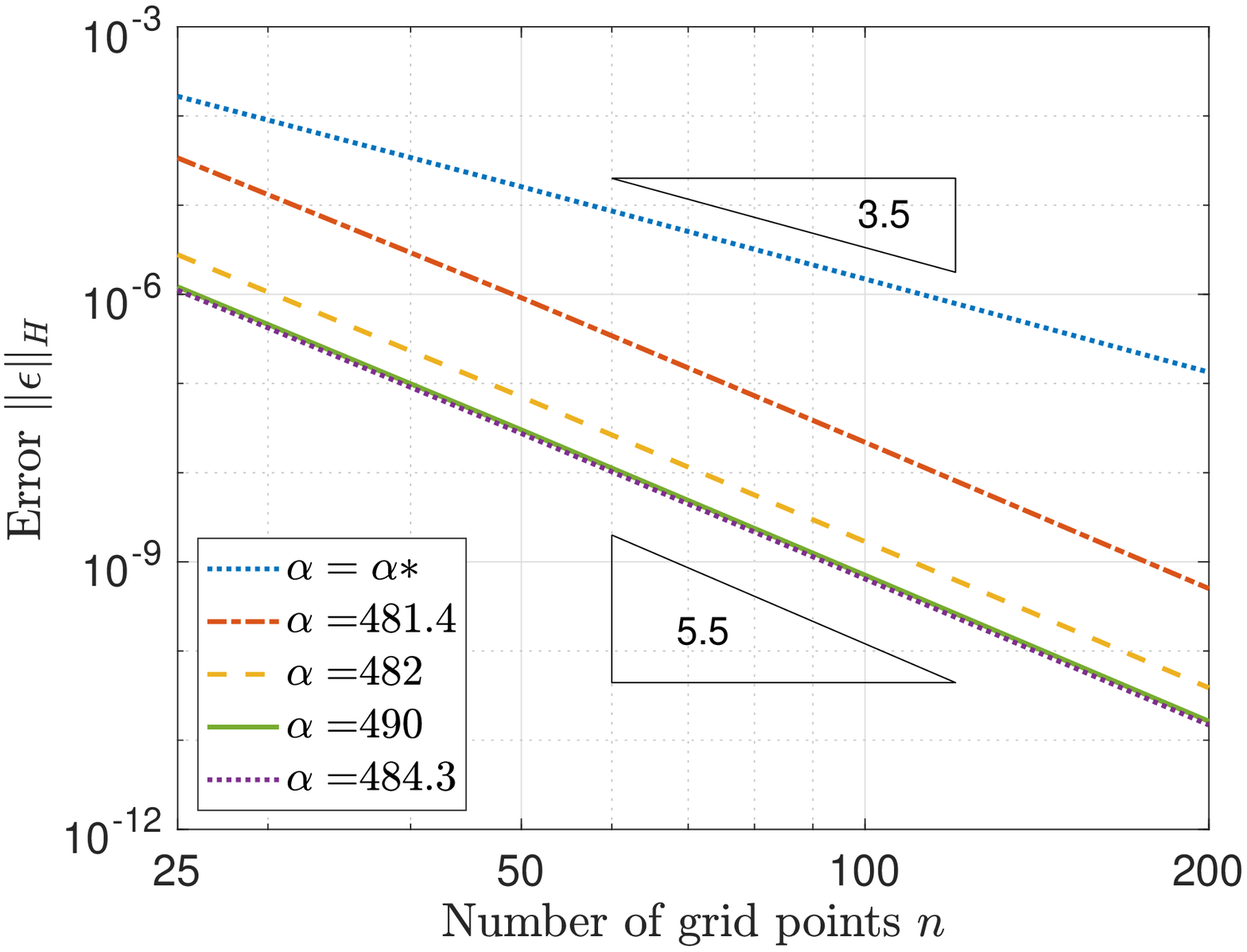}}
\caption{The error as a function of time, for various choices of $\free$. Here with $c=3$. }
\label{fig:HeatNeumannN40}
\end{figure}
In Figure~\ref{fig:HeatNeumannN40}(a), we show the resulting error $\|\fel\|_H$, as a function of time. 
We see that if we choose the free parameter $\free<\limit$, such that $A$ has negative eigenvalues, the error grows as a consequence of the lack of an energy estimate. If we instead choose $\free=\limit$, where $A$ has one (in practice two) additional zero eigenvalue(s), it is still possible to achieve an energy estimate. However, the convergence rate is decreased compared to the choices $\free>\limit$, see Figure~\ref{fig:HeatNeumannN40}(b). 
Since the errors vary in time, the convergence rates shown in Figure~\ref{fig:HeatNeumannN40}(b) have been computed using average errors. Note that compared to the logarithmic scaling it is difficult to see, but the error is 10\% smaller when using $\free=484.3$ instead of $\free=490$, which is consistent with the results from the time-independent case.

\subsection{The wave equation}

Consider the wave equation $u_{tt}=u_{xx}$ in the domain $x\in [0,1]$ and $t\in [0,2]$. The boundary condition is either Dirichlet on both sides or Neumann on both sides, with the boundary data obtained by the manufactured solution $u=\cos(2\pi x+1)\cos(2\pi t+2)$. The semi-discrete approximation is obtained by the SBP-SAT method presented in Section \ref{sec_discrete}. We note that the same discretization matrices, $\DDN$, are also used to discretize Poisson's equation in Sections~\ref{SubSecNeumann}-\ref{SubSecDirichlet}, where the accuracy and spectrum properties of $\DDN$ are analyzed. 

For the time discretization of the wave equation, explicit time integrators are advantageous for computational efficiency. We choose the classical Runge-Kutta method. When verifying accuracy of the semi-discrete approximation, the time step $dt$ is chosen small enough so that the error in the numerical solution is dominated by the spatial discretization. However, the time step restriction for stability is determined by the CFL condition $dt\leq ch$, where the constant $c$ depends on $\rho(\DDN)$, i.e. the spectral radius of the discretization matrix $\DDN$. 
For the Neumann problem, $\rho(\DN)$ is plotted in Figure \ref{NeumannAccuracyAndSpectrum}(b) as a function of $\free$ and is small when $\free< 484.6$. For the Dirichlet problem, the spectral radius $\rho(\DD)$ is plotted in Figure \ref{Fig7_SpectralRadiusD2tildeDirichlet} as a function of $\free$, and takes the smallest value when $\free\approx 487.3$ for $\factor=1$. We note that for a fixed $\free$, the spectral radius increases when $\factor$ increases.


We solve the wave equation with 31 grid points in space ($n=30$), and plot the error in the SBP norm with different values of $\free$ and $\factor$
in Figure \ref{fig:Wave_D_N}. To account for the time evolution, we plot the arithmetic mean of the errors from all time steps.
\begin{figure}[ht]\centering
\includegraphics[width=.6\textwidth]
{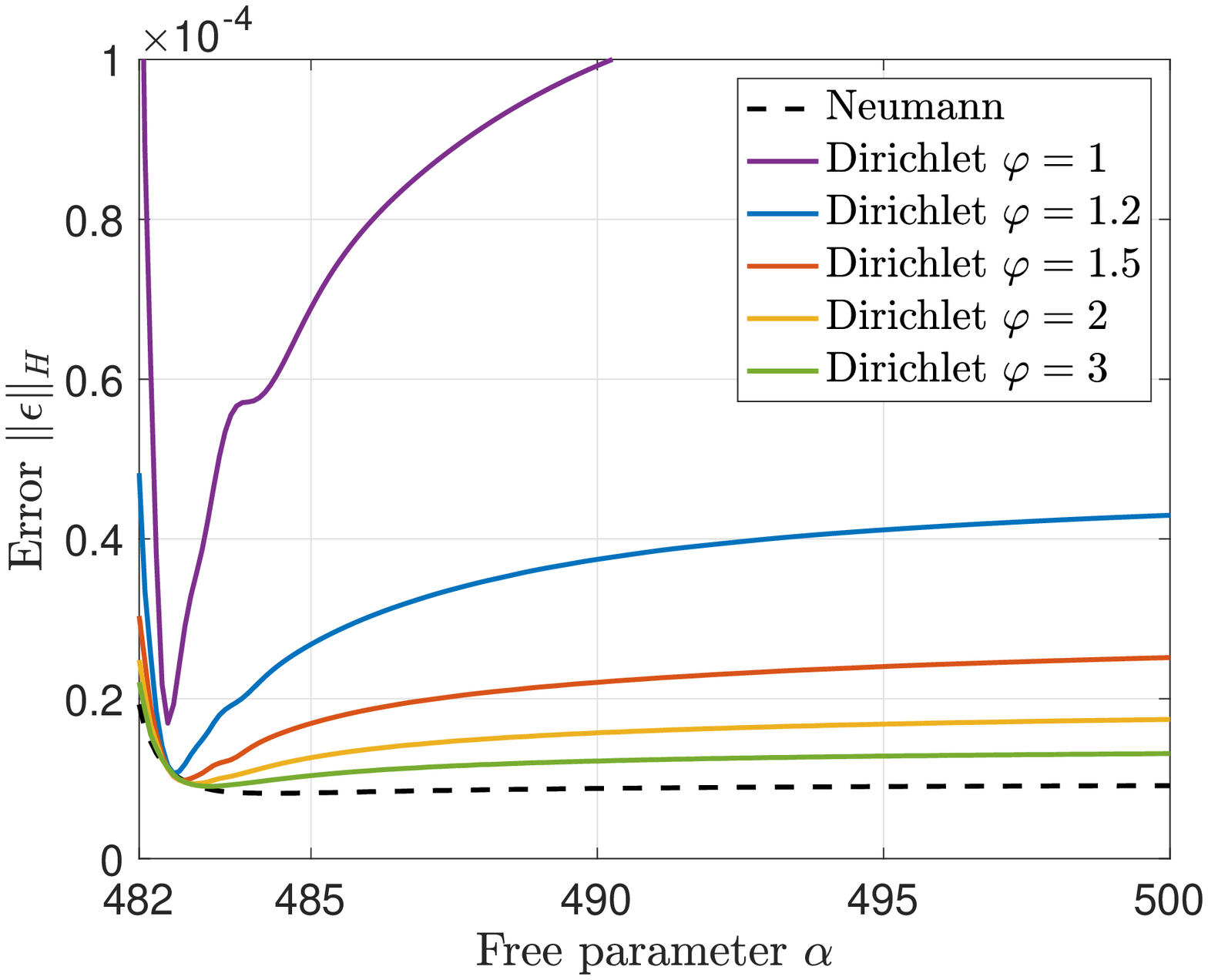}
\caption{$\|\fel\|_H$ as a function of $\free$, for various values of $\factor$.}
\label{fig:Wave_D_N}
\end{figure} 
We observe that for a fixed $\free$, a larger $\factor$ gives a smaller error. For a fixed $\factor$, the error is smallest when $\free\in [482.44, 484.30]$.
This error behavior is very similar to what was observed for Poisson's equation in Figure~\ref{D2tildeDirichletAcc}(a), although there it was not possible to use $\factor=1$.
Thus the conflict between spectral radius and accuracy is the same as was seen for Poisson's equation. A good balance can be obtained by using the recommended combination of $\free$ and $\factor$ in Section \ref{sec_rec}.

\section{Conclusions}

For the SBP-SAT approximation of the second derivative, the discretization matrix is singular for equations with Neumann boundary conditions only. The first main contribution of this paper is to derive an analytical expression of the Moore--Penrose inverse for that singular discretization matrix. This can be used to solve Poisson's equation and certain discretizations of time-dependent problems. For the second and fourth order schemes, we have proved that the discretization matrix is rank deficient by one, which is often assumed in previous works. 

In addition, we have constructed a one-parameter family of the sixth order accurate SBP operator for the second derivative, with a particular choice of the parameter reducing to the original operator in \cite{Mattsson2004503}. 
We have seen that it is possible to tune this free parameter such that the discretization matrix is rank deficient by {\it more than} one, indicating that the assumption of rank deficiency by one is not guaranteed per se. On the other hand, the free parameter also offers a possibility of improvement.
Considering different equations and boundary conditions, we have performed a detailed analysis for the parameter choices to optimize the corresponding SBP operators in terms of stability, accuracy and spectral radius. Our results have shown that improvements can be made by choosing the parameter differently than the original value, especially when having Neumann boundary conditions.





\appendix

\section{Details from Section 5}
\label{Syst2421}

The accuracy and symmetry demands on $A$ from Section~\ref{constructA} yields the following linear system of equations:
\begin{align*}
\scalebox{.870}{$\left[\begin{array}{ccccccccccccccccccccc|c}
1 & 1 & 1 & 1 & 1 & 1 & 0 & 0 & 0 & 0 & 0 & 0 & 0 & 0 & 0 & 0 & 0 & 0 & 0 & 0 & 0 & 0\\ 
 0 & 1 & 2 & 3 & 4 & 5 & 0 & 0 & 0 & 0 & 0 & 0 & 0 & 0 & 0 & 0 & 0 & 0 & 0 & 0 & 0 & -1\\ 
 0 & 1 & 8 & 27 & 64 & 125 & 0 & 0 & 0 & 0 & 0 & 0 & 0 & 0 & 0 & 0 & 0 & 0 & 0 & 0 & 0 & 0\\ 
 0 & 1 & 16 & 81 & 256 & 625 & 0 & 0 & 0 & 0 & 0 & 0 & 0 & 0 & 0 & 0 & 0 & 0 & 0 & 0 & 0 & 0\\ 
 0 & 1 & 0 & 0 & 0 & 0 & 1 & 1 & 1 & 1 & 1 & 0 & 0 & 0 & 0 & 0 & 0 & 0 & 0 & 0 & 0 & 0\\ 
 0 &-1 & 0 & 0 & 0 & 0 & 0 & 1 & 2 & 3 & 4 & 0 & 0 & 0 & 0 & 0 & 0 & 0 & 0 & 0 & 0 & 0\\
 0 & -1 & 0 & 0 & 0 & 0 & 0 & 1 & 8 & 27 & 64 & 0 & 0 & 0 & 0 & 0 & 0 & 0 & 0 & 0 & 0 & 0\\ 
 0 & 1 & 0 & 0 & 0 & 0 & 0 & 1 & 16 & 81 & 256 & 0 & 0 & 0 & 0 & 0 & 0 & 0 & 0 & 0 & 0 & 0\\ 
 0 & 0 & 1 & 0 & 0 & 0 & 0 & 1 & 0 & 0 & 0 & 1 & 1 & 1 & 1 & 0 & 0 & 0 & 0 & 0 & 0 & 0\\ 
 0 & 0 & -2 & 0 & 0 & 0 & 0 & -1 & 0 & 0 & 0 & 0 & 1 & 2 & 3 & 0 & 0 & 0 & 0 & 0 & 0 & 0\\
 0 & 0 & -8 & 0 & 0 & 0 & 0 & -1 & 0 & 0 & 0 & 0 & 1 & 8 & 27 & 0 & 0 & 0 & 0 & 0 & 0 & 0\\
 0 & 0 & 16 & 0 & 0 & 0 & 0 & 1 & 0 & 0 & 0 & 0 & 1 & 16 & 81 & 0 & 0 & 0 & 0 & 0 & 0 & 0\\
 0 & 0 & 0 & 1 & 0 & 0 & 0 & 0 & 1 & 0 & 0 & 0 & 1 & 0 & 0 & 1 & 1 & 1 & 0 & 0 & 0 & \frac{2}{180}\\ 
 0 & 0 & 0 &-3 & 0 & 0 & 0 & 0 & -2 & 0 & 0 & 0 & -1 & 0 & 0 & 0 & 1 & 2 & 0 & 0 & 0 &  \frac{6}{180}\\ 
 0 & 0 & 0 &-27 & 0 & 0 & 0 & 0 & -8 & 0 & 0 & 0 & -1 & 0 & 0 & 0 & 1 & 8 & 0 & 0 & 0 &  \frac{54}{180}\\ 
 0 & 0 & 0 & 81 & 0 & 0 & 0 & 0 & 16 & 0 & 0 & 0 & 1 & 0 & 0 & 0 & 1 & 16 & 0 & 0 & 0 &  \frac{162}{180}\\ 
 0 & 0 & 0 & 0 & 1 & 0 & 0 & 0 & 0 & 1 & 0 & 0 & 0 & 1 & 0 & 0 & 1 & 0 & 1 & 1 & 0 &  \frac{-25}{180}\\ 
 0 & 0 & 0 & 0 & -4 & 0 & 0 & 0 & 0 & -3 & 0 & 0 & 0 & -2 & 0 & 0 & -1 & 0 & 0 & 1 & 0 &  \frac{-48}{180}\\ 
 0 & 0 & 0 & 0 & -64 & 0 & 0 & 0 & 0 & -27 & 0 & 0 & 0 & -8 & 0 & 0 & -1 & 0 & 0 & 1 & 0 & \frac{ -162}{180}\\ 
 0 & 0 & 0 & 0 & 256 & 0 & 0 & 0 & 0 & 81 & 0 & 0 & 0 & 16 & 0 & 0 & 1 & 0 & 0 & 1 & 0 &  \frac{-270}{180}\\ 
 0 & 0 & 0 & 0 & 0 & 1 & 0 & 0 & 0 & 0 & 1 & 0 & 0 & 0 & 1 & 0 & 0 & 1 & 0 & 1 & 1 & \frac{ 245}{180}\\ 
 0 & 0 & 0 & 0 & 0 &-5 & 0 & 0 & 0 & 0 & -4 & 0 & 0 & 0 &-3 & 0 & 0 & -2 & 0 & -1 & 0 &  \frac{222}{180}\\ 
 0 & 0 & 0 & 0 & 0& -125 & 0 & 0 & 0 & 0 & -64 & 0 & 0 & 0 &-27 & 0 & 0 & -8 & 0 & -1 & 0 &  \frac{108}{180}\\ 
 0 & 0 & 0 & 0 & 0 & 625 & 0 & 0 & 0 & 0 & 256 & 0 & 0 & 0 & 81 & 0 & 0 & 16 & 0 & 1 & 0 & 0 
 \end{array}\right] $}
\end{align*}
This system has 24 equations and 21 unknowns,
where the unknowns are ordered as $\aaa$, $\aab$, $\aac$, $\aad$, $\aae$, $\aaf$, $\abb$, $\abc$, $\abd$, $\abe$, $\abf$, $\acc$, $\acd$, $\ace$, $\acf$, $\add$, $\ade$, $\adf$, $\aee$, $\aef$ and $\aff$. Its solution is shown in \eqref{A6}.

Furthermore, in the proof of Proposition~\ref{propfreelimit}, it is referred to Table~\ref{Tab:ConvFree}. The table shows the two values of $\limit$ as a function of $n\leq24$, and how they converge as $n$ increases.
\begin{table}[htbp]
\centering
$\begin{array}{|ccc|}
\hline
\min(\limit)&\max(\limit)&n\\\hline
481.3406894997601&481.3410851822219&n=11\\
481.3408797227131&481.3408949417200&n=12\\
481.3408847406793&481.3408899235506&n=13\\
481.3408871324619&481.3408875317594&n=14\\
481.3408873292311&481.3408873349902&n=15\\
481.3408873299936&481.3408873342276&n=16\\
481.3408873319172&481.3408873323040&n=17\\
481.3408873321098&481.3408873321114&n=18\\
481.3408873321089&481.3408873321123&n=19\\
481.3408873321105&481.3408873321108&n=20\\
481.3408873321106&481.3408873321106&n=21\\
481.3408873321106&481.3408873321106&n=22\\
481.3408873321106&481.3408873321106&n=23\\
481.3408873321106&481.3408873321106&n=24\\
\hline
\end{array}$
\caption{The value of $\limit$ varies slightly as a function of $n$, but has converged for $n\geq21$.}
\label{Tab:ConvFree}
\end{table}

\newpage

\bibliographystyle{plain} 
\bibliography{ref} %


\end{document}